\documentclass{amsart}
\usepackage[fontsize=12pt]{scrextend}
\usepackage{m9macros}
\usepackage{mymacros}
\usepackage{tikz}
\usetikzlibrary{spy}
\usepackage{tkz-euclide}
\usepackage{tikz-3dplot}
\usepackage{graphicx}
\usepackage{float}
\usepackage{enumerate}
\usepackage[english]{babel}
\usepackage{longtable}
\usepackage{tabularx}
\newcolumntype{C}[1]{>{\centering\arraybackslash}p{#1}}
\newcolumntype{L}[1]{>{\raggedright\arraybackslash}p{#1}}
\graphicspath{ {./images/} }
\usepackage[ruled,vlined]{algorithm2e}
\usepackage{cleveref}
\usepackage[utf8]{inputenc}
\usepackage{csquotes}
\usepackage[backend=biber,style=numeric,firstinits=true]{biblatex}
\usepackage{fancyhdr}
\usepackage{amsthm}
\usepackage{nameref}

\usepackage{placeins}
\graphicspath{ {figures/} }
\title[Minimization of the inradius in Minkowski spaces]{Minimization of the inradius of convex bodies for prescribed diameter and circumradius in Minkowski spaces}
\author[R. Brandenberg]{René Brandenberg}
\author[B. González]{Bernardo González Merino}
\author[M. Runge]{Mia Runge}
\begin{document}
 \thanks{
This research has been funded by the PID2022-136320NB-I00 project / AEI/10.13039/501100011033/ FEDER, UE}

\maketitle

\newtheorem{theorem}{Theorem}[section]
\newtheorem{lemma}[theorem]{Lemma}
\newtheorem{proposition}[theorem]{Proposition}
\newtheorem{corollary}[theorem]{Corollary}
\newtheorem{example}[theorem]{Example}
\newtheorem{problem}[theorem]{Problem}
\newtheorem{question}[theorem]{Question}
\theoremstyle{definition}
\newtheorem{definition}[theorem]{Definition}
\newtheorem{remark}[theorem]{Remark}
\newtheorem*{definition*}{Definition}
\newtheorem*{proposition*}{Proposition}
\section*{Abstract}
We study Blaschke-Santaló diagrams for the inradius, circumradius, and diameter in arbitrary-dimensional Minkowski spaces. Their boundary parts are regularly described by four linear inequalities and one more complex inequality, and we analyze which bodies fill the
boundaries and for which (types of) norms that boundaries collapse to a single point. In this context, the diagram with respect to the $1$-norm in 3-space plays an exceptional role. We complete the description of the according diagram by tightening Bohnenblust's inequality through involvement of the inradius. 

\section{Introduction}
\label{sec:intro}

For $n \ge 2$ let $\CK^n$ be the set of \cemph{red}{convex bodies} (\ie~non-empty, convex, and compact sets) in $\R^n$ and $\M^n=(\R^n,\norm{\cdot})$ an \cemph{red}{$n$-dimensional Minkowski space} with the according unit ball $\B$. The  \cemph{red}{Euclidean} space is denoted by $\E^n:=(\R^n,\norm{\cdot}_2)$. For any $K\in\CK^n$, let $R(K)$ be the \cemph{red}{circumradius}, 
$r(K)$ the \cemph{red}{inradius}, and  
$D(K)$ 
the  \cemph{red}{diameter} of $K$. 

The mapping $f_{\M^n}: \CK^n\to [0,1]^2$ is defined by 
\begin{equation*}
     f_{\M^n}(K)=\left(\frac{r(K)}{R(K)}, \frac{D(K)}{2R(K)}\right).
\end{equation*}
Its image $f_{\M^n}(\CK^n)$ is called the \cemph{red}{Blaschke-Santaló diagram} for the inradius, diameter, and circumradius or simply the \cemph{red}{($r,D,R$)-diagram} (\wrt~the Minkowski space $\M^n$). For the sake of brevity, we 
write $f$ instead of $f_{\M^n}$ if we talk about a general Minkowski space $\M^n$ or if $\M^n$ is clear from the context.

Our goal is to understand the general structure of $f(\mathcal K^n)$. It is well known that $f(\mathcal K^n)$ does not have holes \cite[Lemma 2.3]{ngons}, thus it suffices to describe its boundaries. 
This is obtained by providing a \cemph{red}{complete system of inequalities}. Here, we say that a system of valid inequalities is complete if, for every tuple $(r,D,R)$ that fulfills all inequalities, the existence of a $K\in\CK^n$ with $r(K) = r$,  $D(K) = D$, and $R(K) = R$ is guaranteed. A complete system of inequalities for the inradius, diameter, and circumradius for Euclidean 2-space is presented in \cite{santalo} and for Euclidean 3-space in \cite{BGR_euclidean}. General (even non-symmetric) norms for $(r,D,R)$-diagrams are considered (mostly for the planar case) in \cite{ngons}. 

It turns out that the boundary of $f(\mathcal K^n)$ can regularly be described by five geometric inequalities (\cf~\Cref{fig:generalshape}).
The first boundary is the \cemph{red}{symmetry boundary} (upper boundary), induced by the inequality
\begin{equation}\label{eq:upper}
    D(K) \leq 2R(K).
\end{equation}
All symmetric $K$ (and all those who possess an antipodal pair of points in their circumsphere) reach equality in \eqref{eq:upper}. 
Furthermore, it has been shown in \cite[Theorem 1.5]{ngons} that all $K\in\CK^n$ fulfill $D(K)=2R(K)$ if and only if $\B$ is a parallelotope. 

The second part of the boundary is the \cemph{red}{completeness boundary} (right boundary), induced by the \cemph{red}{concentricity inequality} \cite{CaspaniPapini}
\begin{equation}\label{eq:right}
    r(K) + R(K) \leq D(K).
\end{equation}
All (diametrically) complete sets reach equality in \eqref{eq:right}. 

The lower bound on the diameter-circumradius ratio depends on the chosen Minkowski space. The \cemph{red}{Jung constant} 
is defined as $j(\M^n):=\max\{R(K)/D(K), K\in\CK^n\}$.
Thus, the diameter-circumradius ratio lies in the interval $[1/j(\M^n),2]$ and the \cemph{red}{Jung inequality}
\begin{equation}
\label{eq:generalJung}
    R(K)\leq j(\M^n) D(K)
\end{equation}
induces the \cemph{red}{Jung boundary} (lower boundary).

The left-hand side of the diagram is usually split into two parts. 
The first is induced by the trivial inequality $r(K)\geq 0$. It is called the \cemph{red}{lower dimensional sets inequality}, or $LDSI$ for brevity, since it is fulfilled with equality if and only if $K$ is subdimensional. The other part is again described by an inequality that depends on the chosen Minkowski space and an exact expression cannot be given in general. It is called the \cemph{red}{isosceles simplices inequality} or $ISI$ since in Euclidean 2- and 3-space exactly isosceles simplices reach equality \cite{BGR_euclidean,santalo}. Here, we call a simplex \cemph{red}{isosceles} if at most one edge is not of diametrical length and if all edges are of diametrical length, the simplex is called \cemph{red}{diametric}. Since needed in our proves, even for general Minkowski spaces, diametric simplices in Euclidean spaces are called \cemph{red}{regular} and regular $n$-simplices \cemph{red}{centered}, if the Euclidean distance of the vertices to the origin is constant. 

In many cases, such as in $\E^n$ for $n\geq 4$, the $ISI$ is the only unknown 
inequality. Therefore, the main focus of the paper lies in the understanding of the left boundary of these diagrams, or, in other words in the behavior of the \emph{minimization} of the inradius for prescribed values of the circumradius and the diameter.

 Obviously, segments are always mapped to $(0,1)$. The first of our main results concerns in which Minkowski spaces all subdimensional bodies are mapped to $(0,1)$.
\begin{theorem}
\label{thm:rbound}
    Let $\M^n$ be a Minkowski space. 
    Every lower dimensional set $K\in \CK^n$ fulfills $D(K)=2R(K)$ 
    if and only if 
    \begin{enumerate}[i)]
        \item $n=2$ or $\B$ is a parallelotope, or
        \item $n=3$ and $\B$ is a cross-polytope. 
    \end{enumerate}
\end{theorem}
While $i)$ is obvious or well-known \cite{ngons}, $ii)$ states that the 3-dimensional cross-polytope plays a very unique role in this question. Even more, the proof of the theorem reveals an interesting connection to Hanner-polytopes \cite{hanner, hansenlima}. In order to prove $ii)$, we show in \Cref{lem:32ipsymmetric} that $\B$ is a Hanner-polytope (which is always symmetric) if and only if all planar triangles have diameter-circumradius ratio 2 within $\M^n$. 

The 3-dimensional cross-polytope is special not only because of \Cref{thm:rbound}. It also belongs to the exceptional cases of possible unit balls of a Minkowski space where the corresponding Jung inequality becomes redundant (\cf~\Cref{prop:jungboundcollaps}). 
This implies that only three inequalities suffice to completely describe the diagram $f(\CK^3)$ in this setting (c.f.~Figure \ref{fig:c3diagram}). Moreover, the $ISI$ turns out to be linear, which seems to be quite unique, too. 

 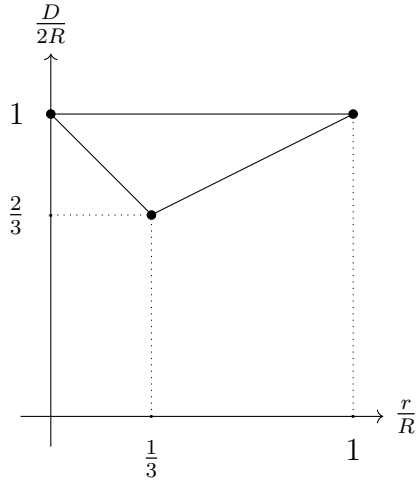
\begin{figure}[H]
    \centering
    \begin{tikzpicture}[scale=4]

     \draw[->] (-0.1,0) -- (1.1,0) node[right] {$\frac{r}{R}$};
     \draw[->] (0,-0.1) -- (0,1.2) node[above] {$\frac{D}{2R}$}
     ;

     \draw[ domain=0:1, smooth, variable=\x, color=black] plot ({\x}, {1})
     ;
     \draw[ domain=0.333:1, smooth, variable=\x] plot ({\x}, {0.5*\x+0.5})
     ;
     \draw[name path=F2, domain=0:0.333, smooth, variable=\x] plot ({\x},{-\x+1} ); 
     \draw[ domain=0:0.333, dotted, variable=\x] plot ({\x}, {0.666})
     ;
       \draw[ domain=0:1, dotted, variable=\y] plot ({1}, {\y})
     ;
      \node[label=below:$1$] (one) at (1,0) {.};
     \draw[ domain=0:0.666, dotted, variable=\y] plot ({0.333}, {\y})
     ;
      \node[label=below:$\frac{1}{3}$] (max) at (0.333,0) {.};
      \node[label=left:$1$] (delta) at (0,1) {.};
      \node[label=left:$\frac{2}{3}$] (rho) at (0,0.666) {.};

     \draw[fill=black] (1,1) circle[radius=0.4pt];

     \draw[fill=black] (0,1) circle[radius=0.4pt];
     \draw[fill=black] (0.333,0.666) circle[radius=0.4pt];

 \end{tikzpicture}
    \caption{The $(r,D,R)$-diagram with a cross-polytopal unit ball 
    in 3-space.}
    \label{fig:c3diagram}
\end{figure}

\begin{theorem}
\label{thm:c3diagram}
    Let $\M^3$ be a Minkowski space whose unit ball $\B$ is a $3$-dimensional cross-polytope and let $K\in \CK^3$. Then, 
    \begin{align*}
 D(K) \leq 2R(K),\quad 
    r(K)+R(K) \leq D(K), \quad\text{and}\quad  2(R(K)-r(K)) \leq D(K).
    \end{align*}
    Moreover, the three inequalities build a complete system for $f(\CK^3)$.
\end{theorem}

For every $D\in[1/j(\M^n),2]$ let $r_D:=\min\{r(K): K \in \CK^n\text{ with }R(K)=1, D(K)=D\}$. We call a point $(r_D,D/2)$, $D\in[1/j(\M^n),2]$, \cemph{red}{left-most} and we say that $K\in\CK^n$ is mapped to a left-most point if $(r(K)/R(K),D(K)/2R(K))=(r_D,D/2)$ for some $D\in[1/j(\M^n),2]$. The left-most points might not completely describe the $ISI$ but since every point in the diagram can be horizontally connected to the completeness boundary, the $ISI$ is fully characterized by the left-most points (\cf~\Cref{rem:horizontalisi}).

One may notice that the set of points $(r(K)/R(K),D(K)/2R(K))$ given by the convex bodies attaining equality in the $ISI$ 
describes a non-increasing curve and the induced boundary is filled by isosceles triangles or simplices in all known $(r,D,R)$-diagrams \cite{ngons,BGR_euclidean, santalo}. Our third result shows that this part of the boundary is always non-increasing and that every left-most point on the $ISI$ is, at the very least, the image of a simplex.

\begin{theorem}
    \label{thm:left-mostpoint-ndim}
    Let $\M^n$ be a Minkowski space.
    \begin{enumerate}[i)]
        \item For every left-most $(r_D,D/2)$, $D\in[1/j(\M^n),2]$, there exists a subdimensional set or a simplex $S$ such that $R(S)=1$, $D(S)=D$, and $r(S)=r_D$. 
        \item  The set of left-most points $\{(r_D,D/2):D\in[1/j(\M^n),2]\}$ is non-increasing, \ie, if $1/j(\M^n)\leq D_1<D_2\leq 2$, then $r_{D_2} \leq r_{D_1}$.
    \end{enumerate}
\end{theorem}
In the planar setting, we additionally show in \Cref{thm:isiplanar} that every left-most point is in fact already the image of an isosceles triangle, explaining the naming $ISI$.

The paper is organized as follows. In Section 2 we collect technical and preliminary results. In Sections 3 and 4, we consider the $LDSI$ and $ISI$ in general, and prove \Cref{thm:rbound} and \Cref{thm:left-mostpoint-ndim}. We conclude the paper by proving \Cref{thm:c3diagram}. 
\section{Preliminaries}

Let $\set{u^1,\dots,u^n}$ denote the 
canonical basis 
of $\R^n$.
For any $X\subset \R^n$, let $\lin(X)$, $\aff(X)$, and $\conv(X)$ be the \cemph{red}{linear}, \cemph{red}{affine}, and \cemph{red}{convex hull} of $X$, respectively. The convex hull of two points $x$ and $y$ is called a \cemph{red}{segment} and is abbreviated by $[x,y]$. The \cemph{red}{boundary} of $X$ is denoted by $\bd(X)$, the \cemph{red}{interior} by $\inte(X)$, the \cemph{red}{closure} by $\cl(X)$, the \cemph{red}{relative boundary} by $\relbd(X)$, and the \cemph{red}{relative interior} by $\relint(X)$.  
For any $X,Y\subset\R^n$ and $\rho\in\R$, let $X+Y:=\{x+y: x\in X, y\in Y\}$ be the \cemph{red}{Minkowski sum} of $X$ and $Y$ and $\rho X:=\{\rho x: x\in X\}$ be the \cemph{red}{$\rho$-dilatation} of $X$. We abbreviate $x+Y:=\set{x}+Y$ and $-X:=(-1)X$.
If $X=-X$, then $X$ is called \cemph{red}{0-symmetric}. 

Let $K \in \CK^n$.
A point $p$ in the boundary of $K$ is \cemph{red}{extreme} if $p\notin \conv(K \setminus \set{p})$. The set of extreme points is denoted by $\ext(K)$.
The \cemph{red}{support function} of $K \in \CK^n$ is defined as $h_K(\cdot): \R^n \to \R $, $h_K(a):=\max_{x\in K}a^{\top}s$ and the  
\cemph{red}{outer normal cone} of $K$ in $v \in \bd(K)$ is denoted by $N_K(v):=\set{a\in\R^n:a^{\top}v = h_K(a) 
}$. 
The \cemph{red}{polar} of $K$ is 
the set $K\pol:=\set{a\in\R^n: h_K(a) \leq1 
}$.

The \cemph{red}{circumradius} of $K \in \CK^n$ is defined as
\[
  R(K):=\inf_{c\in\R^n}\sup_{x\in K} \norm{x-c}=\inf\set{\rho:~K\subset c+\rho\B,~c\in\R^n}
\]
and the \cemph{red}{inradius} as
\[
   r(K):=\sup\set{\rho:~c+\rho\B\subset K,~c\in\R^n}. 
\]
 
The \cemph{red}{diameter} is the length of the maximal distance of two points in $K$:
\begin{equation*}
    D(K):=\max_{x,y\in K}\norm{x-y}=2 \max_{x,y\in K} R([x,y]). 
\end{equation*}
For $s\in\R^n\setminus\set{0}$, the \cemph{red}{$s$-breadth} $b_{s}$ is defined as
                \begin{equation*}
                    b_{s}(K):=  \frac{h_K(s)+h_{K}(-s)}{h_{\B}(s)}.
                \end{equation*}

The following proposition collects well-known properties of the diameter \cite{GritzmannKlee}, including the fact that the diameter can also be described as the maximal breadth of a convex body.
\begin{proposition}
\label{prop:diamproperties}
Let $K\in\CK^n$.
\begin{equation*}
    D(K)=\max_{s\in\R^n\setminus \set{0}} b_{s}(K)=\max_{s\in\ext(\B\pol)} b_{s}(K)=D\left(\frac{K-K}{2}\right)=2R\left(\frac{K-K}{2}\right).
\end{equation*}
\end{proposition}

The three functionals $r,D,R$ are all monotonically increasing with respect to set inclusion and homogeneous of degree $1$. 
Moreover, they fulfill a certain affine invariance, namely: let $A:\R^n\to\R^n$ be a non-singular affine transformation and $L_A$ be the linear transformation corresponding to $A$. Then, denoting with $R',r',D'$ the functionals with respect to the Minkowski space with  unit ball $L_A(\B)$, we have
\begin{equation*}
R'(A(K))=R(K),\quad r'(A(K))=r(K),\quad \text{ and }\quad D'(A(K))=D(K).
\end{equation*}

We use the following notation for \cemph{red}{hyperplanes}: for $a\in\R^n\setminus\set{0}$ and $\beta\in \R$, we define $H^{ =}_{(a,\beta)}:=\set{x\in\R^n:a^{\top}x = \beta}$. (Closed) \cemph{red}{half-spaces} are denoted analogously using "$\leq$" and "$\geq$". 

For $K,C\in \CK^n$ we say that $K$ is \cemph{red}{optimally contained} in $C$, and denote it by $K\optc C$, if $K\subset C$ and there exist no $c\in\R^n$ and $\rho\in(0,1)$ such that $K\subset c+\rho C$. The following proposition characterizes optimal containment  \cite[Theorem 2.3]{NoDimIndep}.
    \begin{proposition}
    \label{prop:opt}
    Let $K,C\in \CK^n$. Then $K\optc C$ if and only if
    \begin{itemize}
        \item[i)] $K\subset C$ and
        \item[ii)] for some $k\in\{2,\hdots,dim(C)+1\}$, there exist $p^1,\hdots,p^k\in \bd(K)\cap \bd(C)$ and half-spaces $H^{\leq}_{(a^{i},(a^i)^{\top}p^i)}$ supporting $C$ at $p^{i}$ with $a^1,\hdots,a^k\in \ext(C^{\circ})\setminus\{0\}$, affinely independent, such that $0\in \relint(\conv(\{a^1,\hdots,a^k\}))$.
    \end{itemize}
    \end{proposition}
To be precise, in \cite{NoDimIndep}, the condition on the outer normals is written as $0\in \conv(\{a^1,\hdots,a^k\})$. However, the case $0\in\relbd(\conv(\{a^1,\hdots,a^k\}))$ would imply that $0$ is already in the convex hull of a proper subset of $\{a^1,\hdots,a^k\}$. It is therefore appropriate to state this here in this form.

In \cite[Lemma 2.3]{ngons} it is shown that $f(\CK^n)$ is star-shaped with respect to $f(\B)=(1,1)$:
\begin{proposition}
\label{prop:starshaped}
Let $K\in \CK^n$ with $K\optc \B$. Then, 
\begin{equation*}
    f((1-\lambda)K+\lambda \B) =(1-\lambda)f(K) +\lambda f(\B),
\end{equation*}
for every $\lambda\in[0,1]$.
\end{proposition}
By \Cref{prop:starshaped}, it is sufficient to describe the boundaries of $f(\CK^n)$ entirely in order to obtain the complete diagram.

There is a special case in which the diagram 
collapses into the single segment $[(0,1),(1,1)]$ \cite{ngons}. 

\begin{proposition}
\label{prop:parallelotopes}
$D(K)=2R(K)$ for all $K\in\CK^n$ if and only if $\B$ is a parallelotope. 
\end{proposition}


We recall a recently proven property showing quasiconcavity of the inradius of simplices sharing a common facet \cite[Theorem 1.2]{BGR_euclidean}.
\begin{proposition}
\label{prop:quasiconc}
Let $\M^n$ be a Minkowski space, $p^2,\ldots,p^{n+1}\in\R^n$ affinely independent, and $P$ a convex set contained in the open half-space bounded by the affine hull of $p^2,\ldots,p^{n+1}$. Define $S_p$, $p \in P$,
to be the simplex with vertices $\set{p,p^2,\ldots,p^{n+1}}$. Then there exists 
$p^* \in \ext(P)$ such that $S_{p^*}$ has minimal inradius $r(S_{p^*})$ over all simplices $S_{p}, p \in P$.
\end{proposition}

We use this property to compare inradii of simplices, whose vertices all belong to the boundary of $\B$. Here, the convex hull $P$ of points on $\bd(\B)$, our choices for the changing vertex, may not belong to $\bd(\B)$ itself. \Cref{cor:bdarea} below shows that we may finally choose the last vertex on a subset of $\bd(\B)$ which belong to the projection of 
$P$ onto $\bd(\B)$ with a center of projection being a vertex of the shared facet $\conv\{p^2,\dots,p^{n+1}\}$ of the simplices.
  
For $\set{q^1,\ldots,q^k}\subset\bd(C)$ and $P:=\set{p^1,\ldots,p^m}\subset\bd(C)$  we say that $q$ belongs to the \cemph{red}{$C,P$-convex hull} of $\set{q^1,\ldots,q^k}$, which we denote by $\conv_{C,P}\left(\set{q^1,\ldots,q^k}\right)$, if 
\begin{equation*}
    q \in \conv \left( \bigcup_{j=1}^{m} \set{ p^j+\sum_{i=1}^{k} \alpha_i(q^i-p^j ):\alpha_i\geq 0, i\in[k], \sum_{i=1}^k \alpha_i\geq 1 }\right) \cap \bd(C). 
\end{equation*}
If $P$ consists of a single point $p$, we abbreviate $\conv_{C,\set{p}} =: \conv_{C,p}$.

Using this notation, the quasiconcavity property can be applied to points on the boundary of a general convex body \cite[Corollary 3.5]{BGR_euclidean}. 

\begin{corollary}  \label{cor:bdarea}
Let $\M^n$ be a Minkowski space and $\set{p^2,\ldots,p^{n+1}}\subset \bd(\B)$ be affinely independent. Furthermore, let $\set{q^1,\ldots,q^k}\subset \bd(\B)$ be contained in the same open half-space bounded by the hyperplane $\aff\left(\set{p^2,\ldots,p^{n+1}}\right)$. Then, for $S:=\conv\left(\set{p^1,p^2,\ldots,p^{n+1}}\right)$ with $p^1 \in \conv_{\B,\set{p^2,\hdots,p^{n+1}}}$
and $S_i:=\conv\left(\set{q^i,p^2,\ldots,p^{n+1}}\right)$
we have 
\begin{equation*}
   r(S)\geq \min_{i\in[k]} r(S_i).
\end{equation*}
\end{corollary}

\section{Well-known parts of the diagram}
\label{sec:knownparts}
We say that one of the five boundaries of the diagram $f(\CK^n)$ (symmetry boundary, completeness boundary, Jung boundary, $LDSI$, and  $ISI$) \cemph{red}{collapses} if 
the system of inequalities is complete even without this (type of) inequality.

The upper bound \eqref{eq:upper} never collapses since segments are mapped to $(0,1)$, the unit ball to $(1,1)$ and the connecting line can be filled by convex combinations of segments and $\B$ by \Cref{prop:starshaped}. 

The right boundary, induced by \eqref{eq:right}, is a linear boundary of $f(\CK^n)$, too. A convex body $K$ is called (diametrically) \cemph{red}{complete} if $D(K\cup \set{p})> D(K)$ for every $p\notin K$. 
A complete set $K^*$ with $K^*\supset K$ and $D(K^*)=D(K)$ is called a \cemph{red}{completion} of $K$. A completion of $K$ with $R(K^*)=R(K)$ is called a \cemph{red}{Scott-completion}. There exists a Scott-completion for every $K\in\CK^n$ independently of the chosen Minkowski space \cite{scott1981sets, vrecica1981note}. In \cite{sallee1986} it is shown that every complete set is mapped to the completeness boundary (which is the reason for our naming), and obviously, for every  $K\in\CK^n$, every one of its Scott-completions is mapped to the rightmost point of the diagram with the same $y$-coordinate as $K$. Using \Cref{prop:parallelotopes}, we easily conclude that this boundary collapses if and only if $\B$ is a parallelotope.

Since the left-hand side of the diagram consists of the parts induced by the $LDSI$ and the $ISI$, for all Minkowski spaces, $f(\CK^n)$ has a shape as shown in \Cref{fig:generalshape}. 

\begin{figure}[ht]
    \centering
      \begin{tikzpicture}[scale=4]

     \draw[->] (-0.1,0) -- (1.1,0) node[right] {$\frac{r}{R}$};
     \draw[->] (0,-0.1) -- (0,1.2) node[above] {$\frac{D}{2R}$}
     ;

     \draw[ domain=0:1, smooth, variable=\x, color=black] plot ({\x}, {1})
     ;
     \draw[ domain=0.5:1, smooth, variable=\x] plot ({\x}, {0.5*\x+0.5})
     ;
     \draw[ domain=0.25:0.5, smooth, variable=\x] plot ({\x}, {0.75})
     ;
     ;
     \draw[name path=F2, domain=0.75:0.9, smooth, variable=\y] plot ({(-100/9) *(\y-0.9)*(\y-0.6)},{\y} ) 
     ;
     ;
      \draw[ domain=0:1, dotted, variable=\y] plot ({1}, {\y})
     ;
      \node[label=below:$1$] (one) at (1,0) {.};
     ;
      \node[label=left:$1$] (delta) at (0,1) {.};

     \node[label=above:$\B$] ($T$) at (1,1) {.};
     \draw[fill=black] (1,1) circle[radius=0.4pt];

     \node[label=above right:$L$] ($L_D$) at (0,1) {.};
     \draw[fill=black] (0,1) circle[radius=0.4pt];


    \node[label=below right:$S_J^{*}$] (Tmax) at (0.5,0.75) {.};
     \draw[fill=black] (0.5,0.75) circle[radius=0.4pt];

     \node[label=below left:$S_J$] (T) at (0.25,0.75) {.};

     \draw[fill=black] (0.5,0.75) circle[radius=0.4pt];
     \draw[fill=black] (0,0.9) circle[radius=0.4pt];
     \node[label=below left:$S_0$] ($T$) at (0,0.9) {.};
     \draw[fill=black] (0.25,0.75) circle[radius=0.4pt];
 \end{tikzpicture}

    \caption{The general shape of $f_{\M^n}(\CK^n)$. The unit ball is mapped to $(1,1)$, segments $L$ to $(0,1)$, and the smallest diameter-circumradius ratio for subdimensional sets is attained by $S_0$. The leftmost point on the Jung-boundary is attained by $S_J$ and its completion, which also lies on the completeness boundary, is denoted by $S_J^{*}$. 
    }
    \label{fig:generalshape}
\end{figure}
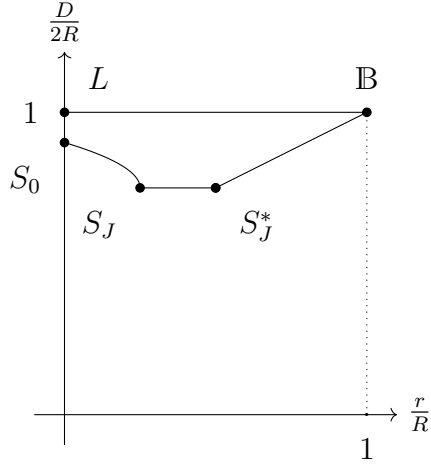

The Jung boundary \eqref{eq:generalJung} is also linear. It is named after Jung \cite{Jung1901}, who proved that $j(\E^n)=\sqrt{n/2(n+1)}$. In \cite[Lemma 2.10 \& Theorem 1.1]{AsymmetryJung} it is proven, as a consequence of Proposition \ref{prop:opt}, that there exists an $n$-simplex or a subdimensional set $S_J$ with $R(S_J)/D(S_J)=j(\M^n)$
such that it has minimum inradius amongst all sets attaining equality in \eqref{eq:generalJung}. If we let $S_J^*$ be a Scott-completion of $S_J$, then $S_J^*$ clearly fulfills \eqref{eq:right} and  \eqref{eq:generalJung} with equality and therefore is of maximal inradius among the Jung-extreme sets. The family of sets $(1-t)S_J+t S_J^*$, $t\in[0,1]$ fills the Jung boundary of $f(\CK^n)$.

Bohnenblust \cite{bohnenblust1938convex} showed that in general
\begin{equation} \label{eq:bohnenblust}
    j(\M^n)\leq \frac{n}{n+1} 
\end{equation}
and the 
the Minkowski spaces in which this inequality is attained with equality are exactly the cases in which the Jung boundary collapses. This collapse is characterised by the following equivalent conditions  \cite[Corollary 2.9]{AsymmetryJung}. 
\begin{proposition}
 \label{prop:jungboundcollaps}
 Let $\M^n$ be a Minkowski space. Then, the following are equivalent.
 \begin{enumerate}[i)]
    \item The Jung-boundary collapses.
     \item $j(\M^n)=\frac{n}{n+1}$, \ie$\M^n$ is Bohnenblust-extreme. 
     \item There exists a complete simplex in $\M^n$.
     \item $S-S\subset D(S)\B \subset (n+1)(S\cap (-S))$,
         where $S$ is a non-singular linear transformation of a centered regular 
         $n$-simplex.
 \end{enumerate}
\end{proposition}

Examples of unit balls where the Jung boundary collapses are 
regular hexagons in the planar case and cross-polytopes or cube-octahedrons in 3-space. 

In the following, we study the types of triangles mapped to the Jung boundary in the planar setting. We call those triangles \cemph{red}{Jung-triangles}. In $\E^n$, the regular $n$-simplex is always mapped to the Jung boundary. Let us mention that it is not known if diametric simplices always exist in all Minkowski spaces for $n\geq 5$. Existence is shown in Minkowski spaces with a unit ball sufficiently close to the Euclidean ball \cite{dekster2000} or the cube 
 \cite{swanepoelvilla2008}.  
Diametric triangles always have a unique completion \cite{brru2023}. Again, this is different in higher dimensions. In $\E^3$, \eg, the regular tetrahedron is obviously diametric, but has (not only) the two famous Meissner bodies \cite{meissner1912} as different completions. 

\begin{theorem}\label{thm:SmoothStrictlydiametric}
 Let $\M^2$ be a Minkowski plane. If the unit ball $\B$ is strictly convex and smooth, every Jung-triangle is diametric.
\end{theorem}
\begin{proof}
Let $S:=\conv(\set{p^1,p^2,p^3})\optc \B$ be a Jung-triangle. Since $\B$ is not a parallelogram we have $D:=D(S) < 2$ and, by \Cref{prop:opt}, there exist outer normals $a^i$ supporting $S$ at $p^i$, $i=1,2,3$,  with $0\in\inte (\conv(\set{a^1,a^2,a^3}))$.  
Hence, moving one vertex slightly does not change the optimal containment. 

The diameter of a triangle is attained by one of its edges. Assume that $D=D([p^1,p^2])>D([p^2,p^3])$. Now, consider $p^1+D\B$. Assume the boundary of $p^1+D\B$ does not cross the boundary of $\B$ at $p^2$. Then, the hyperplane with outer normal $a_2$ also supports $p^1+D\B$ at $p^2$. But since $\B$ is strictly convex and smooth, it needs to support $p^1+D\B$ at $p^1+Dp^2$, which implies $p^1=(1-D)p^2$ and therefore $D=2$, a contradiction.  
Thus, the boundary of $p^1+D\B$ crosses 
    the boundary of $\B$ at $p^2$ and we can move $p^2$ slightly such that $p^2\in \inte (p^1+D\B)$ but still $D>D([p^2,p^3])$. Then the diameter of $[p^1,p^2]$ becomes smaller.  
    If we still have $D=D([p^1,p^3])$, we can apply the same idea and move $p^3$ slightly as well and decrease the diameter. Unless $S$ was diametric, we can decrease the diameter while maintaining optimal containment, which contradicts $S$ being a Jung-triangle. 
\end{proof}

\begin{corollary}
   In 
   every Minkowski plane 
   there exists a Jung-extreme diametric triangle.
\end{corollary}
\begin{proof}
Let $\B$ denote the unit ball of the considered Minkowski plane $\M^2$. Then there exists a sequence $(B_k)_{k\in\N}$ of smooth and strictly convex sets $B_k$ being unit balls of Minkowski planes such that 
$B_k\to \B$ \cite{klee1959some}. By \Cref{thm:SmoothStrictlydiametric} for every $B_k$ there exists a diametric Jung-triangle $S_k\optc B_k$. $(B_k)_{k\in\N}$ is bounded as it is converging and therefore $(S_k)_{k\in\N}$ is bounded as well. By the Blaschke-Selection theorem, there exists a converging subsequence of $(S_k)_{k\in\N}$ , so we may assume $S_k \to S$ for some triangle $S$ which is diametric \wrt $\B$. Let $j_k$ be the Jung constant of the Minkowski plane with unit ball $B_k$. Then, 
$j_k\to j(\M^2)$ and $j_k=\frac{R(S_k)}{D(S_k)} \to \frac{R(S)}{D(S)}$. Thus, $j(\M^2) =  \frac{R(S)}{D(S)}$, which shows that $S$ is a diametric Jung-triangle for $\M^2$. 
\end{proof}

\section{The collapse of the $LDSI$}
\label{subsec:upperleft}

To understand when the $LDSI$ collapses, we need to study intersection properties of the gauge. The \cemph{red}{Helly dimension} $\him(K)$ of $K\in\CK^n$ is the smallest integer $k \in [n]$, with the property that
whenever $(K_i)_{i\in I}$ is a family of translations of $K$, \st~$\bigcap_{i\in I'} K_i=\emptyset$ for all $I'\subset I$ with $|I'|\leq k+1$, then $\bigcap_{i\in I} K_i=\emptyset$. It is well-known that $\him(K)\in\set{1,\hdots,\dim(K)}$ with $\him(K)=1$ if and only if $K$ is a parallelotope \cite{sznagy1954}.

A set $\set{a^1,\hdots,a^{k+1}}$ is \cemph{red}{minimally dependent}, if $S:=\conv\left(\set{a^1,\hdots,a^{k+1}}\right)$ is a $k$-simplex with $0\in\relint(S)$. For a convex body $K$, let $\md(K)$ denote the largest integer $k$ such that there exists a minimally dependent set $\set{a^1,\hdots,a^{k+1}}\subset \ext(K\pol)$. The following proposition is given in \cite{boltyanski1976}.
\begin{proposition}
\label{prop:himmd}
For every convex body $K\in\CK^n$, we have $\him(K)=\md(K)$. 
\end{proposition}

A generalization to the Helly dimension is that of the $(m.k)$-intersection property. 
\begin{definition}
Let $m,k\in\N$ with $2\leq k < m$. A convex body $C\in\CK^n$ has the \cemph{red}{$(m.k)$-intersection property} (\cemph{red}{$(m.k)$-IP}) if the following applies: For every family of $m$ translations of $C$ such that every $k$ of them intersect, all $m$ translations intersect. Bodies that possess the $(3.2)$-IP are called \cemph{red}{Hanner-polytopes} (\Cref{prop:hanner} shows why "polytopes"). 
\end{definition}
Notice that $K\in\CK^n$ always has the $(\infty,\him(K))$-IP. It is well-known that the $(m.2)$-IP for $m>4$ implies the $(4.2)$-IP. 
            

The following lemma shows a connection between intersection properties and the diameter-circumradius ratio.

\begin{lemma}
         \label{lem:32ipsymmetric}
   Let $\M^n$ be a Minkowski space. Then, the following are equivalent. 
   \begin{enumerate}[i)]
       \item $\B$ has the $(m.2)$-IP. 
       \item Any convex hull $S$ of $m$ points fulfills $D(S)=2R(S)$.   
  \end{enumerate}
\end{lemma}
\begin{proof}
     $ii) \Rightarrow i)$: Let $x^1,\hdots,x^m\in \R^n$ and  $B_i:= x^i + \B$, $i\in[m]$, such that $B_i \cap B_j \neq \emptyset$, $1\leq i < j\leq m$. Then, $\norm{x^i-x^j}\leq 2$, $1\leq i < j\leq m$, and for  $S:=\conv(\set{x^1,\hdots,x^m})$ it follows that $D(S)\leq 2$. Now, $ii)$ implies $R(S)\leq 1$ , \ie~there exists a center $c\in\R^n$ s.t. $ S\subset c+ \B$. Hence, $ \norm{x^i -c}\leq 1$, and therefore $c \in  B_i$, $i\in[m]$. Thus, $ B_1 \cap \hdots \cap B_m \neq \emptyset$, which shows that $\B$ has the $(m.2)$-IP.

      $i) \Rightarrow ii)$: Let $S=\conv(\set{x^1,\hdots,x^m})$ fulfill $D(S)=2$, \ie~$\norm{x^i- x^j}\leq 2$ and $(x^i + \B) \cap (x^j + \B) \neq \emptyset$ , $1\leq i < j\leq m$. By $i)$ there exists $ c\in  (x^1 + \B) \cap \hdots \cap (x^m + \B) $ and therefore $x^i \in c+\B $, $i\in[m]$.  Thus, $S \subset c+\B$, which implies $R(S)\leq 1$ and, because of \eqref{eq:upper}, therefore $R(S)=1$. 
\end{proof}

The following proposition collects results by Hanner \cite{hanner} and Hansen and Lima \cite{hansenlima} regarding intersection properties useful for our purpose. 
\begin{proposition} \label{prop:hanner}
Let $C\in\CK^n$.
\begin{enumerate}[i)]
    \item $C$ has the $(3.2)$-IP if and only if  $C$ is obtained from linear independent centered segments by up to $\dim(C)$ many
    operations of (Minkowski) adding the segments or taking the convex hull (which means that $C$ must be a symmetric polytope). 
    \item $C$ has the $(4.2)$-IP if and only if $C$ is a parallelotope.
\end{enumerate}
\end{proposition}
Using the intersection properties we may now characterize when the $LDSI$ collapses.

\begin{proof}[Proof of \Cref{thm:rbound}]
In general, $r(K)=0$ if and only if $K$ is subdimensional.
If $n=2$, only segments are subdimensional and 
all segments are symmetric and therefore fulfill \eqref{eq:upper} with equality.  
If $\B$ is a parallelotope, $D(K)=2R(K)$ for all $K$ by \Cref{prop:parallelotopes}. 


Assuming that the $LDSI$ collapses for $n\geq 4$ implies that all lower dimensional sets, and therefore any convex hull of 4 or less points, optimally contained in $\B$ have diameter 2. By \Cref{lem:32ipsymmetric} this implies that $\B$ has the $4.2$- IP. But then, $\B$ is a parallelotope by \Cref{prop:hanner} $ii)$. 

Similarly, if $n=3$ and the $LDSI$ collapses, we need $D(S)=2$ for all triangles $S\optc \B$. Then, $\B$ has the $(3.2)$-IP by \Cref{lem:32ipsymmetric}, and therefore 
\Cref{prop:hanner} $i)$ implies that it is a parallelotope or a cross-polytope.

It therefore remains to be shown that the LDSI actually collapses if $\B$ is a 3-dimensional cross polytope.
By affine invariance, we may assume that $\B$ is the 
regular cross-polytope $\conv(\set{\pm u^1,\hdots,\pm u^4})$. 
Let $K\optc \B$  
be subdimensional. By \Cref{prop:opt}, there are touching points $p^1,\hdots,p^k\in K\cap \bd(\B)$ with $k\in\set{2,3,4}$ and corresponding outer normals. 
If $k\in\set{2,3}$, we know $D(K)=2R(K)$ since there exists a segment or a planar triangle contained in $K$ with the same circumradius.
Assume $k=4$ and that we need all four touching points. 
The cross-polytope has four pairs of facets. If two of the $p^i$, $i=1,2,3,4$, 
are situated on opposing facets, we have $D(K)=2$. Otherwise, the $p^i$ belong to the relative interiors of four facets such that two of those facets intersect in exactly one vertex (\cf~\Cref{fig:crosspoyltope0}).

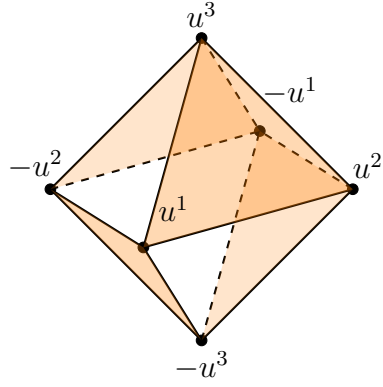
\begin{figure}
    \centering
\begin{tikzpicture}[scale=2]
\def\pA{(1,0,0)}
\def\pB{(-1,0,0)}
\def\pC{(0,1,0)}
\def\pD{(0,-1,0)}
\def\pE{(0,0,1)}
\def\pF{(0,0,-1)}
\def\pGa{(0.2,0.2,0.6)}
\def\pGb{(0.2,0.6,0.2)}
\def\pGc{(0.6,0.2,0.2)}
\def\pHa{(-0.2,-0.2,0.6)}
\def\pHb{(-0.2,-0.6,0.2)}
\def\pHc{(-0.6,-0.2,0.2)}
\def\pIa{(-0.2,0.2,-0.6)}
\def\pIb{(-0.2,0.6,-0.2)}
\def\pIc{(-0.6,0.2,-0.2)}
\def\pJa{(0.2,-0.2,-0.6)}
\def\pJb{(0.2,-0.6,-0.2)}
\def\pJc{(0.6,-0.2,-0.2)}

\def\k{0.6}
\def\km{0.2}
    \draw[thick] \pA -- \pC;
    \draw[thick] \pA -- \pD;
    \draw[thick] \pA -- \pE;
    \draw[thick,dashed] \pA -- \pF;

    \draw[thick] \pB -- \pC;
    \draw[thick] \pB -- \pD;
    \draw[thick] \pB -- \pE;
    \draw[thick,dashed] \pB -- \pF;

    \draw[thick] \pC -- \pE;
    \draw[thick,dashed] \pC -- \pF;

    \draw[thick] \pD -- \pE;
    \draw[thick,dashed] \pD -- \pF;

    \foreach \p in {\pA, \pB, \pC, \pD, \pE, \pF} {
        \filldraw \p circle (1pt);
    }

    \node at \pA [anchor=west] {};
    \node at \pB [anchor=east] {};
    \node at \pC [anchor=south] {};
    \node at \pD [anchor=north] {};
    \node at \pE [anchor=west] {};
    \node at \pF [anchor=east] {};
    \fill[orange, opacity=0.2] \pB -- \pC -- \pF -- cycle; 
      \fill[orange, opacity=0.2] \pA -- \pF -- \pD -- cycle;
      \fill[orange, opacity=0.3] \pB -- \pD -- \pE -- cycle;
      \fill[orange, opacity=0.3] \pA -- \pE -- \pC -- cycle;
\tkzLabelPoint[above](1.1,0,0){$u^2$}
\tkzLabelPoint[above](-1.1,0,0){$-u^2$}
\tkzLabelPoint[above](0,1,0){$u^3$}
\tkzLabelPoint[below](0,-1,0){$-u^3$}
\tkzLabelPoint[above](0.2,0.1,1){$u^1$}
\tkzLabelPoint[above](0.2,0.1,-1){$-u^1$}
\end{tikzpicture}
    \caption{The cross-polytope $\B_1$. 
    The four touching points are situated on non-opposing facets.}
    \label{fig:crosspoyltope0}
\end{figure}

Let us assume, we have a point on $F_1:=\conv(\set{u^1,u^2,u^3})$, $F_2:=\conv(\set{-u^1,-u^2,u^3})$, $F_3:=\conv(\set{u^1,-u^2,-u^3})$, and $F_4:=\conv(\set{-u^1,u^2,-u^3})$. By $H$ we denote the plane containing the quadrangle of touching points and by $H^-$ and $H^+$ the two closed half-spaces induced by $H$. Then, not all three of the vertices of any of the touched facets can be situated in one of the half-spaces $H^-$ or $H^+$. Furthermore, there has to be a pair of vertices of the cross-polytope which is strictly separated by $H$. W.l.o.g., assume $u^3\in \inte(H^+)$ and $-u^3\in \inte(H^-)$. We show that these conditions lead to a contradiction. 

We consider the facet $F_1$. Since $u^3\in \inte(H^+)$, at least one remaining vertex needs to be contained in $H^-$. We may assume that $u^2\in H^-$. \\
Case 1: $u^1,u^2\in H^-$. Then, it follows that $-u^1,-u^2\in H^+$ because of the facets containing $-u^3$, namely $F_3$ and $F_4$. But then, $F_2$ is contained in $H^+$, a contradiction.\\
Case 2: $u^1\in H^+$ and $u^2\in H^-$. Then, $-u^1\in H^+$ because of facet $F_4$. In this case, the condition on $F_2$ 
implies $-u^2\in H^-$. However, this implies that $H$ is the $x_1,x_2$-plane, which contradicts the fact that the touching points are contained in the relative interior of the facets. 
\end{proof}

\section{Properties of the $ISI$}
\label{subsec:left}
The 
$ISI$ collapses if and only if there exists a subdimensional set $K$ with $R(K)=1$ and $D(K)=1/j(\M^n)$.

\begin{theorem}\label{thm:lowerleftcollaps}
Let $\M^n$ be a Minkowski space. 
    If $\him(\B)\leq n-1$, the $ISI$ collapses.
\end{theorem}

\begin{proof}
    If $\him(\B)\leq n-1$, then by \Cref{prop:himmd}, we have $\md(\B)\leq n-1$. Then, for any $K\optc \B$, we can choose at most $n$ affinely independent outer normals as in \Cref{prop:opt}. Thus, the diameter of $K$ is larger than the one of the convex hull of at most $n$ touching points, which is a subdimensional set. Thus, there exists a subdimensional set $K$ with $R(K)=1$ and $D(K)=1/j(\M^n)$ and the set of points provided by the equality case of the $ISI$ collapses.
\end{proof}

Obviously, there are Minkowski spaces with $\him(\B)=n$ (like the Euclidean spaces) such that the $ISI$ does not collapse. However, one cannot replace the \enquote{if} in Theorem \ref{thm:lowerleftcollaps} by an \enquote{if and only if}.

\begin{example}
\label{thm:hellynexample}
    There exist Minkowski spaces $\M^n$ with $\him(\B)=n$ such that the $ISI$ collapses.
\end{example}
\begin{proof}
    We construct an example using the following claim. Let $\B$ fulfill
    \begin{equation}
         \B' \times [-j,j] \subset \B \subset   \B' \times [-1,1],
    \end{equation}
    with $h_{\B}(u_n)=1$ and $j\in[1/(2j(\M')),1]$
for some 0-symmetric $\B'$ with $\dim(\B')=n-1$ and $\M'$ being the $(n-1)$-dimensional Minkowski space induced by $\B'$.
Then, there exists a subdimensional $K$ with $R(K)=1$ and $D(K)=1/j(\M)$ and the $ISI$ collapses.

If we take the Minkowski sum of a convex body $C$ with $\him(C)=n$, such as the Euclidean ball, with a segment $L$, we know $\him(C+L)=\max\set{\him(C),\him(L)}=n$. If the segment $L$ is long enough, $C+L$ fulfils the condition of \eqref{eq:prismlikecondition}.

Now, we prove the above claim. 
Let $K'\optc \B'$ be a Jung-extreme body in $\M'$. Then, 
$R(K'\times \set{0})=1$ and $D(K'\times \set{0})=1/j(\M')$. 

We now show that $j(\M')=j(\M)$. Let $K\optc \B$. 
First, assume we need outer normals not contained in $\lin(\B')$ to describe optimal containment as in \Cref{prop:opt}. 
Then, $K \cap \cl(\B\setminus (\B' \times [-1,j]))$ and $K \cap \cl(\B\setminus (\B' \times [-j,1]))$ are non-empty. 
Taking the breadth in the direction of the $n$-th unit vector, we obtain  
\begin{equation*}
\label{eq:prismlikecondition}
    D(K)\geq b_{u^n}(K)= h_K(u^n)+h_K(-u^n)\geq 2j \geq  \frac{1}{j(\M')}
    . 
\end{equation*} 

Now assume outer normals from $\lin(\B'))$ suffice. 
Denoting the projection of $X$ onto a linear subspace $L$ by $X|L$ we obtain
$R(K|{\lin(\B')})=1$ and therefore
\begin{align*}
    D(K) \geq \max_{s \in \lin(\B')\setminus \set{0}} b_s(K) 
    = D(K|{\lin(\B')}) 
    \geq \frac{1}{j(\M')}.
\end{align*}
However, since $R(K')/D(K')=1/j(\M')$ we finally conclude $j(\M)=j(\M')$, which proves the claim. 

\end{proof}

\begin{remark} \label{rem:horizontalisi}
    Let us notice that the set of left-most points might not be 
    connected. 
   This, for example, can be seen from \cite[Theorem 1.4]{ngons}. There, a Minkowski plane is considered with a regular hexagon as the unit sphere. In this setting $D(K)<2R(K)$ is only possible if $4r(K)> R(K)$. 
    Thus, in this setting, the left-most point with $D(K)=2R(K)$ is not connected to the left-most points with $D(K)<2R(K)$.  
        However, since one can connect every left-most point horizontally to its Scott-completions, the entire left boundary of the diagram is described by the left-most points and horizontal segments.  
    By considering $k$-gons with increasing $k$ we may create more such examples with discontinuities in the left-most points. For example, one may consider the regular $12$-gon with vertices $v^0,\hdots,v^{11}$. 
    Here, we obtain a horizontal segment on the $ISI$ 
   which is filled by a family of isosceles triangles $T_\lambda:=\conv\left(\left\{\frac{1}{2}(v^0+v^1), \lambda v^6 +(1-\lambda)v^5, \lambda v^7 +(1-\lambda)v^8 \right\}\right)$ for $\lambda\in[0,\tfrac{1}{2}]$, \cf~\Cref{fig:12gon}, all having a diameter $1+\tfrac{\sqrt{3}}{2}$ and decreasing inradius.
    (The computation of the exact inequality would go beyond the scope of this 
   paper and is therefore left out.)

    \begin{figure}[ht]
        \centering
\begin{tikzpicture}[scale=0.8]
    \tkzDefPoint(0,0){O}

\tkzDefPoint(-0.5,-0.6){p0}
\tkzDefPoint( 0.5,-0.6){p1}

\tkzDefRegPolygon[side, sides=12, name=V](p0,p1)

\tkzDrawPolygon(V12,V1,V2,V3,V4,V5,V6,V7,V8,V9,V10,V11)

\tkzLabelPoint[below](p0){$v^0$}
\tkzLabelPoint[below](p1){$v^1$}
\tkzLabelPoint[above](V6){$v^5$}
\tkzLabelPoint[above](V7){$v^6$}
\tkzLabelPoint[above](V8){$v^7$}
\tkzLabelPoint[above](V9){$v^8$}

\tkzDefMidPoint(V6,V7) \tkzGetPoint{V67}
\tkzDefMidPoint(V8,V9) \tkzGetPoint{V89}
\tkzDefMidPoint(V1,V2) \tkzGetPoint{V12a}
\tkzDrawPolygon[dgreen](V12a,V67,V89)
\tkzDrawPolygon[dgreen](V12a,V6,V9)

\tkzDrawPoints[fill=black](V12,V1,V2,V3,V4,V5,V6,V7,V8,V9,V10,V11)
\tkzDrawPoints[fill=black](V12a, V67,V89)

\end{tikzpicture}

        \caption{The family of triangles described in \Cref{rem:horizontalisi}, which are mapped to a horizontal segment on the left boundary.}
        \label{fig:12gon}
    \end{figure}
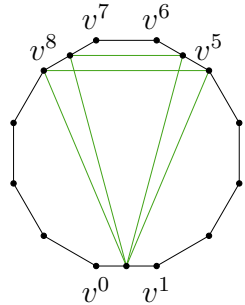
    
    
    
\end{remark}

\FloatBarrier
Next, we study the bodies mapped to left-most points. We denote the largest circumradius-diameter ratio attained by subdimensional sets by $j_{n-1}$. 
To do so,in this section we simply say that a set $X\subset  \bd(\B)$ is open 
if there is an open subset $V$ of $\R^n$ with $V\cap \bd(\B)=X$ and the interior of $X$ is (re-)defined as $\inte(X):=\set{x\in X: (x+\epsilon \B) \cap \bd(\B)\subset X \text{ for some } \epsilon>0}$. This notion is analogous to the definition of the relative interior in affine subspaces.
The boundary is then defined as $\bd(X)=\cl(X)\setminus\inte(X)$. 
Furthermore, $X$ is \cemph{red}{path-connected}, if for any pair of points $x,y\in X$ there exists a continuous path $\gamma:[0,1]\to X$ with $\gamma(0)=x$ and $\gamma(1)=y$.
The following lemma shows that we may connect any point in the diagram, which is induced by a simplex optimally contained in the unit ball, to a point induced by an optimally contained set 
with diameter at least $\frac{1}{j_{n-1}}$, while not increasing the inradius. This is done by continuously moving one vertex until we reach such a diameter 
while keeping optimal containment. 
\begin{lemma}
    \label{lem:connectsimplices}
    Let $S:=\conv(\set{p^1,p^2,\hdots,p^{n+1}})\optc \B$ be a full-dimensional simplex with 
    $D(S)<1/j_{n-1}$. 
    Then, for every $D\in[D(S),1/j_{n-1})$ there exists a simplex $S_D$ with $D(S_D)=D$, $R(S_D)=1$, and $r(S_D)\leq r(S)$.  
\end{lemma}

\begin{proof}
   We show that there is a continuous path $\gamma:[0,1]\mapsto \bd(\B)$ with $\gamma(0)=p^1$ such that 
    \begin{align*}
        &S_t\optc \B,\\
        &r(S_t)\leq r(S),\text{ and}\\
        &D(S_1)\geq \frac{1}{j_{n-1}},
    \end{align*}
   where $S_t:=\conv(\set{\gamma(t),p^2,\hdots,p^{n+1}})$, $t\in[0,1]$.
   The claim then follows by continuity of the diameter. 

    By the assumptions on $S$ we have $\set{p^1,p^2,\hdots,p^{n+1}}\subset \bd(\B)$ and outer normals $a^i$ to $\B$ in 
    $p^i$, $i\in\set{2,\hdots,n+1}$, with $0\in \inte(\conv(\set{a^1,a^2,\hdots,a^{n+1}}))$.  
     We define
\begin{align*}
    Q&:= \set{p\in \bd(\B): \pos(\set{-a^2,\hdots,-a^{n+1}})\cap N_{\B}(p)\neq\emptyset },\\
    Q_{\bd}&:= \set{p\in Q: \bd(\pos(\set{-a^2,\hdots,-a^{n+1}}))\cap N_{\B}(p)\neq\emptyset },\\
    Q^>&:= \set{p\in Q: r(\conv(\set{p,p^2,\hdots,p^{n+1}})>r(S) }, \text{ and}\\
    Q^{\leq}&:= \set{p\in Q: r(\conv(\set{p,p^2,\hdots,p^{n+1}})\leq r(S) }.\\
\end{align*}
Then, $Q$ is closed and path-connected and 
for every $p \in Q$ the set $\conv(\set{p,p^2,\hdots,p^{n+1}})$ is optimally contained in $\B$. If $p\in Q_{\bd}$, there exists an outer normal at $p$ contained in $\bd(\pos(\set{-a^2,\hdots,-a^{n+1}}))$. As explained in the remark after \Cref{prop:opt}, this implies that we may remove at least one of the vertices while preserving optimal containment. In that case the diameter is at most the diameter of a subdimensional set, \ie~
\begin{equation*}
D(\conv(\set{p,p^2,\hdots,p^{n+1}}))\geq 1/j_{n-1}.
\end{equation*}
Hence, $p^1\notin Q_{\bd}$, \ie~$N_{\B}(p^1)\setminus\set{0}\subset \inte(\pos(\set{-a^1,\hdots,-a^{n+1}}))$, and obviously $p^1\in Q^{\leq}$.
Thus, it suffices to show that there exists a
continuous path $\gamma:[0,1]\mapsto Q^{\leq}$ with $\gamma(0)=p^1$ and $\gamma(1)
\in Q_{\bd}$. 

Assume this is not the case. Denoting
\begin{equation*}
    X:=\set{p\in Q^{\leq}: p\text{ and } p^1 \text{ are path-connected in $Q^{\leq}$}},
\end{equation*}
the assumption may be rewritten as $X\cap Q_{\bd}=\emptyset$.

Let $q\in \bd(Q)$. 
There is a sequence of points in $\bd(\B)$ converging to $q$ with no outer normals in $\pos(\set{-a^1,\hdots,-a^{n+1}})$. Then, $q$ has to have an outer normal in $\bd(\pos(\set{-a^2,\hdots,-a^{n+1}}))$. 
Hence, $\bd(Q)\subset Q_{\bd}$, which implies $X\subset\inte(Q)$.

The path-connected components of $Q^{\leq}$ are closed because $Q^{\leq}$ is closed. Thus, $X$ and $Q^{\leq}\setminus{X}$ are closed and disjoint. Furthermore, $Q_{\bd}$ is closed. Together, we obtain that $Q_{\bd}\cup (Q^{\leq}\setminus X)$ and $X$ are two disjoint compact sets. Thus, there is an $\epsilon>0$ such that $A := (X+\epsilon\B)\cap \bd(\B)$ and $Q_{\bd}\cup (Q^{\leq}\setminus X)$ are disjoint and since $X\subset\inte(Q)$ we may choose $\epsilon$ small enough such that $A 
\subset Q$. 
Now, any point in $A \setminus X$ obviously does not belong to $Q^{\le}$ either, which implies
$\bd(A)\subset Q^>$ and therefore $p^1\in \inte(A)$. 
However, this means that for any two point $q^1,q^2 \in \bd(A)$ such that $p^1\in\conv_{\B,p^i}(\set{q^1,q^2})$, for some $i \in \set{2,\dots,n+1}$ (which one may obtain, e.g., by choosing $q^1 \in \bd(A)$ arbitrarily and $q^2\in\bd(A)\cap \aff(\set{p^1,p^i, q^1})$) we have
$q^1,q^2\in Q^>$, $p^1\in Q^{\leq}$, and therefore
\begin{equation*}
   \min\set{
   r(\conv(\set{q^1,p^2,\hdots,p^{n+1}})), r(\conv(\set{q^2,p^2,\hdots,p^{n+1}}))
   } >r(S), 
\end{equation*}
contradicting the quasiconcavity property in \Cref{cor:bdarea}.

\end{proof}

Using this lemma, we now prove \Cref{thm:left-mostpoint-ndim}. 

\begin{proof}[Proof of \Cref{thm:left-mostpoint-ndim}]
    \begin{enumerate}[i)]
        \item Let $K\optc\B$ be mapped to a left-most point. For $D(K)\geq \frac{1}{j_{n-1}}$ this immediately implies that $K$ is subdimensional. Thus, let us assume $D(K)< \frac{1}{j_{n-1}}$. By \Cref{prop:opt}, there exist touching points $p^1,\hdots, p^{n+1}\in K\cap\bd(\B)$, %
         such that $S:=\conv(\set{p^1,\hdots, p^{n+1}})\optc \B$,  $D(S)\leq D(K)<\frac{1}{j_{n-1}}$,  and $r(S)\leq r(K)$.
        By \Cref{lem:connectsimplices} there exists a simplex $S_{D(K)}\optc\B$ with $r(S_{D(K)})\leq r(S)\leq r(K)$ and $D(S_{D(K)})=D(K)$. Since $K$ was mapped to a left-most point, we obtain $r(S_{D(K)})=r(K)$ and $S_{D(K)}$ is mapped to the same point as $K$.
        \item Assume there are bodies $K_1,K_2\optc\B$ mapped to left-most points with $\frac{1}{j_{n-1}}>D(K_1)\geq D(K_2)$ and $r(K_1)> r(K_2)$. Then, by 
        $i)$, we may assume that both, $K_1$ and $K_2$, 
        are simplices. 
        Now, we apply \Cref{lem:connectsimplices} to $K_2$, to show that there is a simplex $S_{D(K_1)}\optc\B$ with $D(S_{D(K_1)})=D(K_1)$ and $r(S_{D(K_1)})\leq r(K_2)<r(K_1)$. This contradicts $K_1$ being mapped to a left-most point.
    \end{enumerate}
\end{proof}


Restricting to the planar setting, we prove two additional results. We show that the left-most points are 
attained by segments and triangles only and that we may always choose the triangles to be isosceles.
\begin{theorem}
\label{thm:isiplanar}
Let $\M^2$ be a Minkowski plane. 
\begin{enumerate}[i)]
    \item  If a body $K\in\CK^2$ is mapped to a left-most point, it is a segment or a triangle.
    \item For each left-most point, there is a segment or an isosceles triangle mapped onto it. 
\end{enumerate}
\end{theorem}

\begin{proof}
    \begin{enumerate}[i)]
        \item Let $K\in\CK^2$ be a body that is mapped to a left-most point. If $D(K)=2R(K)$, then $r(K)=0$ and therefore it has to be 
        a segment. 
        
        Now, assume $D(K)<2R(K)$. Then, by \Cref{prop:opt}, there exists a proper triangle $S \subset K$ 
        with vertices $p^1,p^2,p^3 \in\bd(K)\cap \bd(c+R(K)\B)$, 
        and corresponding outer normals $a^1,a^2,a^3$ with $0\in\inte\left(\conv\left(\set{a^1,a^2,a^3}\right)\right)$
    such that $R(S) = R(K)$,
        $D(S)\leq D(K)$, and $r(S)\leq r(K)$
        . We show that  
        assuming $K \supsetneq S$ then $r(K)>r(S)$, which would contradict $K$ being a left-most point by \Cref{thm:left-mostpoint-ndim} $ii)$.     
      
For $i=1,2,3$ let $H_i$ be the line which contains the edge of $S$ opposite of $p^i$ and $b^i$ the outer normal of that edge.
If all three of these lines, $H_1,H_2$, and $H_3$, support $K$ as well, we obviously have $K=S$. Moreover, since $K$ is convex, $H_i$, $i=1,2,3$, either supports $K$ or intersects $\bd(K)$ in exactly two 
the vertices 
of $S$ different from $p^i$.%

  We may assume w.l.o.g.~that $\B$ is the inball of $S$, \ie
\begin{equation*}
    \B\optc S\optc K \optc c+R(K)\B 
\end{equation*}
for some circumcenter $c \in \R^n$.
If two vertices, $p^i$ and $p^j$, $1\leq i<j\leq 3$, of $S$ would be contained in the inball $\B$, it follows $[-p^i,-p^j] \subset \B\subset S$ by the symmetry of $\B$. However, this is not possible for an edge of a triangle. Thus, at most one vertex of $S$ is contained in $\B$. 

Now we distinguish two cases of how the endpoints of an edge of $S$ that is not contained in the boundary of $K$ behave. 

Case 1: If 
a line $H_i$, not supporting $K$, is such that no endpoint of the corresponding edge of $S$ is contained in $\B$, we may move $\B$ slightly in the direction of the outer normal of that line, bringing it into
the interior of $K$. 
However, this then proves $r(S)<r(K)$. 

Case 2: If there is no such line, we may assume $p^1\in \B$ and that $H_2$ does not support $K$. As we have argued that no two vertices of $S$ can belong to $\B$, $H_1$ needs to support $K$, otherwise $[p^2,p^3]$ is an edge not supporting $K$ with both endpoints not in $\B$ (\cf~\Cref{fig:triangleleftmost}). 

Moreover, both, $H_2$ and $H_3$ support $\B$ in  $p^1$. Now, consider the common outer normal $a^1$ of $S$, $\B$, and $c+R(K)\B$ in $p^1$.
We see that $a^1\in \pos( \set{b^2,b^3})$ and all directions in $ \pos( \set{b^2,b^3})$ are outer normals of $\B$ in $p^1$, too. Since the outer normals of $\B$ at $p$ are identical to the ones of $c+R(K)\B$ at $c+R(K)p$ we obtain that if $a^1\in \inte(\pos( \set{b^2,b^3}))$, then $p^1$ is the unique optimizer in 
such a direction. Thus, in that case all directions in $\pos( \set{b^2,b^3})$ are also outer normals of $c+R(K)\B$ in $p^1$. However, we would obtain that all directions in $\pos( \set{b^2,b^3})$ are also outer normals of $K$, contradicting that $H_2$ does not support $K$.
Assuming $a^1=b^2$ (up to scaling) would also contradict $H_2$ not supporting $K$. Thus, the only possibility that remains is $a^1=b^3$ (up to scaling). Now, $a^1$ and $a^2$ are both outer normals of $K$ and $c+R(K)\B$ at $p^2$ and since $0\in\conv\left(\set{a^1,a^2,a^3}\right)$, 
the same must hold for $-a^3$. However, this then implies that $D(K)=b_{a^3}(K)=2R(K)$, which contradicts our assumption at the begin of the proof. 
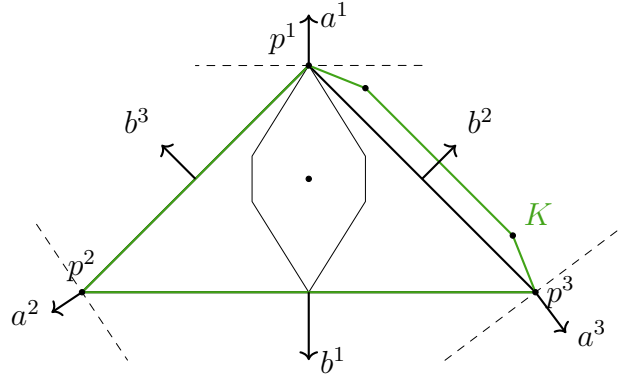
\begin{figure}
    \centering
  \begin{tikzpicture}[scale=1.5]

\tkzDefPoint(0,0){O}

\tkzDefPoint(0,1){p1}
\tkzDefPoint(-2,-1){p2}
\tkzDefPoint(2,-1){p3}
\tkzDrawPolygon[black,thick](p1,p2,p3)

\tkzDefPoint(0.5,0.2){v1}
\tkzDefPoint(-0.5,0.2){v2}
\tkzDefPoint(-0.5,-0.2){v3}
\tkzDefPoint(0.5,-0.2){v4}
\tkzDefPoint(0,-1){mp1}
\tkzDrawPolygon[black](mp1,v3,v2,p1,v1,v4)

\tkzDefPoint(0.5,0.8){k1}
\tkzDefPoint(1.8,-0.5){k2}
\tkzDrawPolygon[dgreen, thick](p1,k1,k2,p3,p2)

\tkzDrawPoints[fill=black](p1,p2,p3,O,k1,k2)

\tkzDefMidPoint(p3,p2)\tkzGetPoint{m1} 
\tkzDefMidPoint(p1,p3)\tkzGetPoint{m2} 
\tkzDefMidPoint(p2,p1)\tkzGetPoint{m3} 

\tkzDefLine[perpendicular=through m1](p3,p2)\tkzGetPoint{n1}
\tkzDefLine[perpendicular=through m2](p1,p3)\tkzGetPoint{n2}
\tkzDefLine[perpendicular=through m3](p2,p1)\tkzGetPoint{n3}

\draw[->,thick] (m1) -- ($(m1)!0.15!(n1)$) node[right] {$b^{1}$};
\draw[->,thick] (m2) -- ($(m2)!0.15!(n2)$) node[above right] {$b^{2}$};
\draw[->,thick] (m3) -- ($(m3)!0.15!(n3)$) node[above left] {$b^{3}$};

\tkzDefPoint(0,1.5){a1}
\tkzDefPoint(-2.3,-1.2){a2}
\tkzDefPoint(2.3,-1.4){a3}

\draw[->,thick] (p1) -- ($(p1)!0.9!(a1)$) node[right] {$a^{1}$};
\draw[->,thick] (p2) -- ($(p2)!0.9!(a2)$) node[left] {$a^{2}$};
\draw[->,thick] (p3) -- ($(p3)!0.9!(a3)$) node[right] {$a^{3}$};
\tkzLabelPoint[above left](p1){$p^1$}
\tkzLabelPoint[above](p2){$p^2$}
\tkzLabelPoint[right](p3){$p^3$}

\tkzDefLine[perpendicular=through p1](p1,a1)
\tkzGetPoint{h1}
\tkzDrawLine[add = 2 and 1, dashed](p1,h1)

\tkzDefLine[perpendicular=through p2](p2,a2)
\tkzGetPoint{h2}
\tkzDrawLine[add = 2 and 1, dashed](p2,h2)

\tkzDefLine[perpendicular=through p3](p3,a3)
\tkzGetPoint{h3}
\tkzDrawLine[add = 2 and 1, dashed](p3,h3)

\tkzLabelPoint[dgreen, above right](k2){$K$}

\end{tikzpicture}
    \caption{Proof of \Cref{thm:isiplanar} $i)$, case 2. The inball touches one vertex and its opposite edge of the triangle of touching points.}
    \label{fig:triangleleftmost}
\end{figure}

\item  Let $S\optc\B$ be mapped to a left-most point. If $D(S)=2$, $S$ is a segment, so we may assume $D(S)<2$ and, by i), that it is a triangle with vertices $p^1,p^2,p^3$. Additionally, we assume that the triangle is not isosceles and w.l.o.g.~that the diameter is attained between $p^1$ and $p^2$
, \ie~$\max\set{\norm{p^1-p^3},\norm{p^2-p^3}}<\norm{p^1-p^2}$. 
From \Cref{prop:opt}, we know that $a^3 \in  \pos(\set{-a^1,-a^2})$ and since 
we are in the planar case, this implies $p^3\in \pos(\set{-p^1,-p^2})$, too.  
Additionally, if $p\in \bd(\B)\cap \pos(\set{-p^1,-p^2})$, then $\conv(\set{p^1,p^2,p})\optc\B$, since there exists an outer normal of $p$ in $\pos(\set{-a^1,-a^2})$. Now, we consider the function which maps points $p \in \bd(\B)\cap \pos(\set{-p^1,-p^2})$ to $\max \set{D([p^1,p]),D([p^2,p])}$ which is continuous as $p$ moves along the boundary. For $p=p^3$, we attain a value strictly smaller than $D(S)$ since we assumed $S$ not to be isosceles, and if $p\in\set{-p^1,-p^2}$, it maps to 2, which is strictly larger than $D(S)$. By the intermediate value theorem there exist $p_{*}^1$ and $p_{*}^2$ with $p_{*}^i\in \bd(\B)\cap \pos(\set{-p^i,p^3})$ such that $\conv(\set{p^1,p^2,p_{*}^i}))$, $i=1,2$, are isosceles triangles with the same diameter and circumradius as $S$. By \Cref{cor:bdarea}, we know that at least one of them has at most the inradius of $S$ and is therefore mapped to the same left-most point as $S$.



    \end{enumerate}
\end{proof}

\section{A complete system for the 3-dimensional cross-polytope}
\label{sec:c3diagram}
\Cref{thm:rbound} and \Cref{prop:jungboundcollaps} show that the diagram of the 1-norm in 3-space (\ie~the unit ball is the 3-dimensional unit cross-polytope $\B_1$) is unique. 
It is (up to linear transformation) the unique non-planar Minkowski space where the $LDSI$ collapses, it 
belongs to the special spaces, where the Jung inequality collapses, and it has a linear $ISI$, which is exceptional, too. Altogether, the diagram becomes a triangle, i.e.~it is also the only polytopal diagram with non-collapsing $ISI$ that we are aware of. In this section, we prove the claimed form of the $ISI$ for this setting. 
\FloatBarrier
\begin{lemma}
\label{lem:areavertices}
    Let $S:=\conv\left(\set{p^1,p^2,p^3,p^4}\right)\optc \B_1$ be a simplex with $D(S)<2$. Then, 
\begin{equation*}
    1-\frac{D(S)}{2}\leq |p^i_j| \leq D(S)-1,
\end{equation*}
    $i\in\set{1,2,3,4}$ and $j\in\set{1,2,3}$.
\end{lemma}

\begin{figure}
    \centering
\begin{tikzpicture}[scale=2.5]
\def\pA{(1,0,0)}
\def\pB{(-1,0,0)}
\def\pC{(0,1,0)}
\def\pD{(0,-1,0)}
\def\pE{(0,0,1)}
\def\pF{(0,0,-1)}
\def\pGa{(0.2,0.2,0.6)}
\def\pGb{(0.2,0.6,0.2)}
\def\pGc{(0.6,0.2,0.2)}
\draw \pGa -- \pGb -- \pGc --cycle;
\def\pHa{(-0.2,-0.2,0.6)}
\def\pHb{(-0.2,-0.6,0.2)}
\def\pHc{(-0.6,-0.2,0.2)}
\draw \pHa -- \pHb -- \pHc --cycle;
\def\pIa{(-0.2,0.2,-0.6)}
\def\pIb{(-0.2,0.6,-0.2)}
\def\pIc{(-0.6,0.2,-0.2)}
\draw \pIa -- \pIb -- \pIc --cycle;
\def\pJa{(0.2,-0.2,-0.6)}
\def\pJb{(0.2,-0.6,-0.2)}
\def\pJc{(0.6,-0.2,-0.2)}
\draw \pJa -- \pJb -- \pJc --cycle;

\def\k{0.6}
\def\km{0.2}
    \draw[thick] \pA -- \pC;
    \draw[thick] \pA -- \pD;
    \draw[thick] \pA -- \pE;
    \draw[thick,dashed] \pA -- \pF;

    \draw[thick] \pB -- \pC;
    \draw[thick] \pB -- \pD;
    \draw[thick] \pB -- \pE;
    \draw[thick,dashed] \pB -- \pF;

    \draw[thick] \pC -- \pE;
    \draw[thick,dashed] \pC -- \pF;

    \draw[thick] \pD -- \pE;
    \draw[thick,dashed] \pD -- \pF;

    \foreach \p in {\pA, \pB, \pC, \pD, \pE, \pF} {
        \filldraw \p circle (1pt);
    }

    \node at \pA [anchor=west] {};
    \node at \pB [anchor=east] {};
    \node at \pC [anchor=south] {};
    \node at \pD [anchor=north] {};
    \node at \pE [anchor=west] {};
    \node at \pF [anchor=east] {};
    \fill[orange, opacity=0.2] \pB -- \pC -- \pF -- cycle; 
      \fill[orange, opacity=0.2] \pA -- \pF -- \pD -- cycle;
      \fill[orange, opacity=0.3] \pB -- \pD -- \pE -- cycle;
      \fill[orange, opacity=0.3] \pA -- \pE -- \pC -- cycle;

\end{tikzpicture}
    \caption{If four points lie on non-opposing facets of the cross-polytope $\B_1$ with a given diameter, they have to be situated in an arrangement like on the orange-colored facets and there inside the small triangles.  }
    \label{fig:crosspoyltope2}
\end{figure}
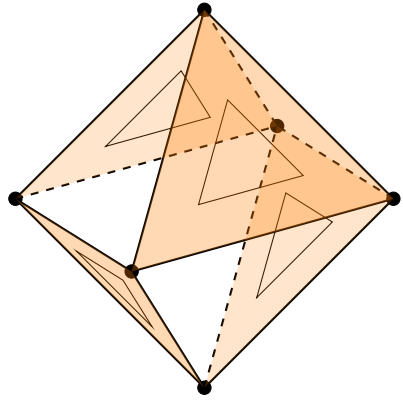

\begin{proof}
   Since $D(S)<2$, all points in $\set{p^1,p^2,p^3,p^4}$ belong to $\bd(\B_1)$, but no 
  two of them can belong to the same nor to opposing facets of $\B_1$. 
   Thus, we may assume
    \begin{align*}
    \label{eq:crosspolyvertices}
        p^1 \in \conv\left(\set{u^1,u^2,u^3}\right), \qquad 
        & p^2 \in \conv\left(\set{-u^1,-u^2,u^3}\right),\\
        p^3 \in \conv\left(\set{-u^1,u^2,-u^3}\right), 
        \qquad 
        & p^4 \in \conv\left(\set{u^1,-u^2,-u^3}\right)
    \end{align*}
(\cf~\Cref{fig:crosspoyltope2}). By \Cref{prop:diamproperties}, $D(S)=\max_{a\in \ext\B_1\pol)}b_a(S)$). Hence, we obtain
\begin{equation}
    \begin{split}
          D(S)&\geq b_{(\phantom{-}1,\phantom{-}1,\phantom{-}1)^\top}(S)
          \geq 1-p^i_1-p^i_2-p^i_3, \quad 
          i\in \set{2,3,4},\\
    D(S)&\geq b_{(-1,-1,\phantom{-}1)^\top}(S) 
    \geq 1+p^i_1+p^i_2-p^i_3, \quad 
    i\in \set{1,3,4},\\
    D(S)&\geq b_{(-1,\phantom{-}1,-1)^\top}(S)
    \geq 1+p^i_1-p^i_2+p^i_3, \quad 
    i\in \set{1,2,4},
    \text{ and }\\
    D(S)&\geq b_{(\phantom{-}1,-1,-1)^\top}(S)
    \geq 1-p^i_1+p^i_2+p^i_3, \quad 
    i\in \set{1,2,3}.
    \end{split}
    \label{eq:crosspolyineqs1}
\end{equation}
Here, the second inequality in the first row uses the fact that 
\begin{equation*}
    b_{(1,1,1)^\top}(S)\geq (1,1,1)^\top p^1 + (-1,-1,-1)^\top p^i=1-p^i_1-p^i_2-p^i_3, \quad i\in \set{2,3,4}
\end{equation*}
and the second inequalities in the other rows follow analogously. 
Now, $p^i\in\bd(\B_1)$ implies $\sum_{k=1}^3|p^i_k|=1$, $i\in\set{1,2,3,4}$. Thus, \eg~for $i=1$, we obtain $p^1_2+p^1_3=1-p^1_1$ and therefore, by the fourth equation in \eqref{eq:crosspolyineqs1} that $D(S)\geq 2-2p^1_1$, which proves $p^1_1\geq 1-\frac{D(S)}{2}$. All other cases again follow analogously.  

Furthermore, by adding the first two inequalities in \eqref{eq:crosspolyineqs1}, we obtain $2D(S)\geq 2-2p^i_3$ or $D(S)-1\geq -p^i_3=|p^i_3|$ for $i\in\set{3,4}$ and again all other cases follow analogously. 



\end{proof}

 \begin{theorem}
 \label{thm:c3newineq}
     For $(\R^3,\norm{\cdot}_1)$ and $K\in\CK^3$, we have
    \begin{equation*}
      R(K)-r(K)\leq \frac{D(K)}{2}
    \end{equation*}
 \end{theorem}
 \begin{proof}
 If $D(K)=2R(K)$, the inequality is obviously fulfilled. Hence, we may assume $K\optc \B_1$ and $D(K) < 2$. 
First, this implies, due to \Cref{prop:opt} and Theorem \Cref{thm:rbound}, that
there must be four touching points $\tilde p^1,\tilde p^2,\tilde p^3, \tilde p^4 \in K \cap \bd(\B_1)$ and the simplex $\tilde S:=\conv\left(\set{\tilde p^1,\tilde p^2,\tilde p^3, \tilde p^4}\right)$ is optimally contained in $\B_1$ again, with $r(\tilde S)\leq r(K)$ and $D:=D(\tilde S)\leq D(K)$. Since $D<2$, we may assume (as in \Cref{lem:areavertices}) 
\begin{equation}
 \label{eq:pscoordinates1}
    \begin{split}
         \tilde p^1 \in \conv\left(\set{u^1,u^2,u^3}\right), \qquad 
        &  \tilde p^2 \in \conv\left(\set{-u^1,-u^2,u^3}\right),\\
         \tilde p^3 \in \conv\left(\set{-u^1,u^2,-u^3}\right), 
        \qquad 
        &  \tilde p^4 \in \conv\left(\set{u^1,-u^2,-u^3}\right).
    \end{split}
\end{equation}


By \Cref{lem:areavertices} we have 
\begin{equation}
\label{eq:pscoordinates2}
    1-\frac{D}{2}\leq |\tilde p^i_j| \leq D-1
\end{equation}
for all  $i\in\set{1,2,3,4}$ and $j\in\set{1,2,3}$ and by \Cref{prop:diamproperties}, the diameter is attained at one of the breadths $b_a$ with $a$ being an outer normal of a facet of $\B_1$. Hence, for any choice of $p^1$, $p^2$, $p^3$, and $p^4$ fulfilling the conditions in  \eqref{eq:pscoordinates1} and \eqref{eq:pscoordinates2}, the diameter of $S:=\conv(\set{p^1,p^2,p^3,p^4})$ is at most $D$ and $S\optc \B_1$. 

By \Cref{prop:quasiconc}, the inradius of $\tilde S$ is at least that of such a simplex $S$ with all vertices 
also being one of the vertices of the triangular areas defined by \eqref{eq:pscoordinates2}. Thus, it suffices to consider these configurations and, using the symmetries of $\B_1$, it turns out that there are only six of them, which can be found in \Cref{table:configs}. There, the projections onto the $x_1,x_2$--plane of the triangular areas on the facets of $\B_1$ are depicted. The proof for the inradii of the six configurations can be found in \nameref{appendixA}. The vertices of the small triangular areas always have two coordinates of absolute value $1-\frac{D}{2}$ and one with absolute value $D-1$. 
We may assume $p^1=(1-\frac{D}{2}, 1-\frac{D}{2}, D-1)^{\top}$. 

First, we consider the case with two vertices whose coordinates differ only in the signs of the two components having absolute values $1-\frac{D}{2}$. Let this happen between $p^1$ and $p^2$. Then, $p^2=( -(1-\frac{D}{2}), -(1-\frac{D}{2}), D-1)^{\top}$. 
Under these assumptions, we have four choices for $p^3$ and $p^4$ leading to different configurations, as shown in the first four rows of \Cref{table:configs}. In three of these cases, the inradius is $1-\frac{D}{2}$ and in the remaining case it is $\frac{D(2-D)}{4-D}$, which is 
lower bounded by $1-\frac{D}{2}$ for all $D \in[\frac{4}{3},2]$. 

Now, assume that there are no two vertices whose coordinates differ only in the signs of the two components having absolute values $1-\frac{D}{2}$. In this case, we may set  $p^2=(-(1-\frac{D}{2}), -(D-1), 1-\frac{D}{2})^{\top}$. 
Moreover, because of our assumption on the signs of the coordinates, $p^4\neq (1-\frac{D}{2}, -(D-1), -(1-\frac{D}{2}))^{\top}$.

Case 1: If $p^4=(D-1, -( 1-\frac{D}{2}), -(1-\frac{D}{2}))^{\top}$ then, because of rotational symmetry around the axis $\lin(\set{(-1,1,-1)^{\top}})$, 
any choice of $p^3$ leads to the same inradius (\cf~the fifth row of \Cref{table:configs}). Here, the inradius is $\frac{D(7D^2-20D+16)}{2(11D^2-24D+16)}$, which is again 
lower bounded by $1-\frac{D}{2}$ for all $D\in[\frac{4}{3},2]$. 

Case 2: If $p^4=(1-\frac{D}{2}, -( 1-\frac{D}{2}), -(D-1))^{\top}$ then, by our assumptions on the coordinates, $p^3=(-(D-1), 1-\frac{D}{2}, -(1-\frac{D}{2}))^{\top}$ or $p^3=(-(1-\frac{D}{2}), D-1, -(1-\frac{D}{2}))^{\top}$. 

The first choice of $p^3$ corresponds to Case 1 (now, the axis $\lin(\set{(1,1,1)^{\top}})$ would correspond to the the axis $\lin(\set{(-1,1,-1)^{\top}})$ in that case). The second choice corresponds to the configuration, which is depicted in the sixth row of \Cref{table:configs}. This configuration has inradius $1-\frac{D}{2}$. 

Hence, the smallest inradius attained is $1-\frac{D}{2}$. It follows 
\begin{align*}
    r(K) &\geq r(\tilde{S})\geq r(S)\geq R(S)-\frac{D(S)}{2}\\
    &\geq R(\tilde S)-\frac{D(\tilde S)}{2}\geq R(K)-\frac{D(K)}{2}.
\end{align*}
\end{proof}

\begin{proof}[Proof of \Cref{thm:c3diagram}]
By affine invariance we may assume $\B=\B_1$. The first two inequalities are \eqref{eq:upper} and $\eqref{eq:right}$. The third is the one proven in \Cref{thm:c3newineq}. These three inequalities describe the entire boundary of $f(\CK^3)$. First, for any segment $L$, $f(L)=(0,1)$. Second, $f(\B_1)=(1,1)$ for the cross-polytope $\B_1$. Their convex combinations fill the points in between, \cf~\Cref{prop:starshaped}. Third, the regular simplex $T$ with vertices at the centers of non-opposing facets is mapped to $(1/3,2/3)$, and the convex combinations of $T$ and $\B_1$ fill the points between $(1/3,2/3)$ and $(1,1)$. 

Finally, defining $S_D=\conv\left(\set{p^1,p^2,p^3,p^4}\right)$ with 
     \begin{equation*}
            p^1=\begin{pmatrix}
                1-\tfrac{D}{2} \\ 1-\tfrac{D}{2} \\ D-1
            \end{pmatrix} \quad
             p^2=\begin{pmatrix}
               -( 1-\frac{D}{2}) \\ -(1-\frac{D}{2}) \\ D-1
            \end{pmatrix} \quad
             p^3=\begin{pmatrix}
                -(1-\frac{D}{2}) \\ 1-\frac{D}{2} \\ -(D-1)
            \end{pmatrix} \quad
             p^4=\begin{pmatrix}
                1-\frac{D}{2} \\ -(1-\frac{D}{2}) \\ -(D-1)
            \end{pmatrix}
        \end{equation*}
        for $D\in\left[\frac{4}{3},2\right]$ (\cf~Configuration 1 in \Cref{table:configs}) attain equality in \Cref{thm:c3newineq}, ranging from $f(S_{4/3})=(1/3,2/3)$ to $f(S_2)=(0,1)$.
    
    By Proposition \ref{prop:starshaped}, the three given inequalities build a complete system of inequalities for $f(\CK^3)$. 
 \end{proof}

    There are four equality cases in \Cref{table:configs} for the $ISI$ in \Cref{thm:c3newineq}. Note that configuration 1 in  \Cref{table:configs} has two shorter edges and is not isosceles. The other three are isosceles, and configuration 6 is even a diametric simplex. 

\newpage

 \begin{table}[H]
 \begin{center}
\begin{tabular}{ |m{4cm}|m{4cm}|m{7cm}| } 
 \hline
 Configuration & Inradius & Incenter \\ 
 \hline
 
 \centering \begin{tikzpicture}[scale=1.3]

\coordinate (O) at (0,0);

\draw[] (0.3,0.3) -- (1,0.3) -- (0.3,1) -- cycle;

\draw[]  (0.3,-0.3) -- (1,-0.3) -- (0.3,-1)  -- cycle;

\draw[]  (-0.3,0.3) -- (-1,0.3) -- (-0.3,1) -- cycle;

\draw[]  (-0.3,-0.3) -- (-1,-0.3) -- (-0.3,-1) -- cycle;

\fill[orange] (0.3,0.3) circle (2pt) ;
\fill[orange] (-0.3,-0.3) circle (2pt) ;
\fill[orange] (-0.3,0.3) circle (2pt) ;
\fill[orange] (0.3,-0.3) circle (2pt) ;
   \draw[->] (-1.2,0) -- (1.2,0) node[right] {$x_1$};
     \draw[->] (0,-1.2) -- (0,1.2) node[above] {$x_2$};

\end{tikzpicture} & $1-\frac{D}{2}$ & $\begin{pmatrix}
     0 \\ 0\\ 0 
 \end{pmatrix}$  \\ 
 \hline
    \centering \begin{tikzpicture}[scale=1.3]

\coordinate (O) at (0,0);

\draw[] (0.3,0.3) -- (1,0.3) -- (0.3,1) -- cycle;

\draw[]  (0.3,-0.3) -- (1,-0.3) -- (0.3,-1)  -- cycle;

\draw[]  (-0.3,0.3) -- (-1,0.3) -- (-0.3,1) -- cycle;

\draw[]  (-0.3,-0.3) -- (-1,-0.3) -- (-0.3,-1) -- cycle;

\fill[orange] (0.3,0.3) circle (2pt) ;
\fill[orange] (-0.3,-0.3) circle (2pt) ;
\fill[orange] (-1,0.3) circle (2pt) ;
\fill[orange] (0.3,-0.3) circle (2pt) ;
   \draw[->] (-1.2,0) -- (1.2,0) node[right] {$x_1$};
     \draw[->] (0,-1.2) -- (0,1.2) node[above] {$x_2$};

\end{tikzpicture} & $1-\frac{D}{2}$ & $\begin{pmatrix}
     \frac{(3D-4)(D-2)}{2D} \\ 0\\ \frac{(3D-4)(D-1)}{D} 
 \end{pmatrix}$  \\ 
 \hline
 \centering\begin{tikzpicture}[scale=1.3]

\coordinate (O) at (0,0);

\draw[] (0.3,0.3) -- (1,0.3) -- (0.3,1) -- cycle;

\draw[]  (0.3,-0.3) -- (1,-0.3) -- (0.3,-1)  -- cycle;

\draw[]  (-0.3,0.3) -- (-1,0.3) -- (-0.3,1) -- cycle;

\draw[]  (-0.3,-0.3) -- (-1,-0.3) -- (-0.3,-1) -- cycle;

\fill[orange] (0.3,0.3) circle (2pt) ;
\fill[orange] (-0.3,-0.3) circle (2pt) ;
\fill[orange] (-0.3,1) circle (2pt) ;
\fill[orange] (1,-0.3) circle (2pt) ;
   \draw[->] (-1.2,0) -- (1.2,0) node[right] {$x_1$};
     \draw[->] (0,-1.2) -- (0,1.2) node[above] {$x_2$};

\end{tikzpicture}& $\frac{D(2-D)}{4-D}$ & $\begin{pmatrix}
     -\frac{(3D-4)(D-2)}{8-2D} \\ -\frac{(3D-4)(D-2)}{8-2D}\\ \frac{3D-4}{4-D} 
 \end{pmatrix}$  \\ 
 \hline
  \centering
\begin{tikzpicture}[scale=1.3]

\coordinate (O) at (0,0);

\draw[] (0.3,0.3) -- (1,0.3) -- (0.3,1) -- cycle;

\draw[]  (0.3,-0.3) -- (1,-0.3) -- (0.3,-1)  -- cycle;

\draw[]  (-0.3,0.3) -- (-1,0.3) -- (-0.3,1) -- cycle;

\draw[]  (-0.3,-0.3) -- (-1,-0.3) -- (-0.3,-1) -- cycle;

\fill[orange] (0.3,0.3) circle (2pt) ;
\fill[orange] (-0.3,-0.3) circle (2pt) ;
\fill[orange] (-1,0.3) circle (2pt) ;
\fill[orange] (1,-0.3) circle (2pt) ;
   \draw[->] (-1.2,0) -- (1.2,0) node[right] {$x_1$};
     \draw[->] (0,-1.2) -- (0,1.2) node[above] {$x_2$};

\end{tikzpicture} & $1-\frac{D}{2}$ & $\begin{pmatrix}
     0 \\ 0\\ \frac{3D-4}{2} 
 \end{pmatrix}$  \\ 
 \hline
 \centering\begin{tikzpicture}[scale=1.3]

\coordinate (O) at (0,0);

\draw[] (0.3,0.3) -- (1,0.3) -- (0.3,1) -- cycle;

\draw[]  (0.3,-0.3) -- (1,-0.3) -- (0.3,-1)  -- cycle;

\draw[]  (-0.3,0.3) -- (-1,0.3) -- (-0.3,1) -- cycle;

\draw[]  (-0.3,-0.3) -- (-1,-0.3) -- (-0.3,-1) -- cycle;

\fill[orange] (0.3,0.3) circle (2pt) ;
\fill[orange] (-0.3,-1) circle (2pt) ;
\fill[orange] (-0.3,1) circle (2pt) ;
\fill[orange] (1,-0.3) circle (2pt) ;
   \draw[->] (-1.2,0) -- (1.2,0) node[right] {$x_1$};
     \draw[->] (0,-1.2) -- (0,1.2) node[above] {$x_2$};

\end{tikzpicture}  & $\frac{D(7D^2-20D+16)}{2(11D^2-24D+16)}$ & $\frac{3D-4}{11D^2-24D+16}\begin{pmatrix}
     2D^2-5D+4 \\ \frac{-(3D-4)(2-D)}{2}\\ \ 2D^2-5D+4
 \end{pmatrix}$  \\ 
 \hline

 \centering\begin{tikzpicture}[scale=1.3]

\coordinate (O) at (0,0);

\draw[] (0.3,0.3) -- (1,0.3) -- (0.3,1) -- cycle;

\draw[]  (0.3,-0.3) -- (1,-0.3) -- (0.3,-1)  -- cycle;

\draw[]  (-0.3,0.3) -- (-1,0.3) -- (-0.3,1) -- cycle;

\draw[]  (-0.3,-0.3) -- (-1,-0.3) -- (-0.3,-1) -- cycle;

\fill[orange] (0.3,0.3) circle (2pt) ;
\fill[orange] (-0.3,-1) circle (2pt) ;
\fill[orange] (-0.3,1) circle (2pt) ;
\fill[orange] (0.3,-0.3) circle (2pt) ;
   \draw[->] (-1.2,0) -- (1.2,0) node[right] {$x_1$};
     \draw[->] (0,-1.2) -- (0,1.2) node[above] {$x_2$};

\end{tikzpicture} & $1-\frac{D}{2}$ & $\begin{pmatrix}
     0 \\ 0\\ 0 
 \end{pmatrix}$  \\
 \hline
\end{tabular}
\end{center}
\caption{The six configurations of simplices considered in the proof of \Cref{thm:c3diagram}.}
\label{table:configs}
\end{table}

\newpage

\section*{References}
\printbibliography[heading=none]

René Brandenberg -- 
Technical University of Munich, School of Computation, Information and Technology, Department of Mathematics, Germany. \\
\textbf{rene.brandenberg@tum.de}

Bernardo González Merino -- 
University of Murcia, Faculty of Computer Science, Department of Engineering and Technology of Computers, Area of Applied Mathematics, Spain.\\
\textbf{bgmerino@um.es}

Mia Runge -- 
Technical University of Munich, School of Computation, Information and Technology, Department of Mathematics, Germany. \\
\textbf{mia.runge@tum.de}

\newpage
\appendix
\section*{Appendix A}
\label{appendixA}
\begin{proof}[Proof of \Cref{table:configs}]
   To prove for a given configuration $S=\conv\left(\set{p^1,p^2,p^3,p^4}\right)$ from \Cref{table:configs} that the inradius is $r$ and that the incenter is $c$, we need to show $c+r\B_1 \subset S$ and that this containment is optimal. To do so, we have computed the inradii by exploiting the symmetries of $\B_1$ as well as $S$. We thus omit here the details and only leave the resulting values of the inradius and the incenter so that the reader can check their validity. 
   For each configuration, we list the six vertices of $c+r\B_1$ as convex combinations of 
   $p^1,p^2,p^3,p^4$, which shows the containment $c+r\B_1 \subset S$. For the optimal containment, we show that some of those vertices are indeed a convex combination of at most three of the vertices of the simplex and  
   provide 
   common outer normals at some of those vertices at the boundary of $S$ with 0 in their convex hull as described in \Cref{prop:opt}. 

    \textit{Configuration 1:}\\
    Vertices:
        \begin{equation*}
            p^1=\begin{pmatrix}
                1-\frac{D}{2} \\ 1-\frac{D}{2} \\ D-1
            \end{pmatrix}, \quad
             p^2=\begin{pmatrix}
               -( 1-\frac{D}{2}) \\ -(1-\frac{D}{2}) \\ D-1
            \end{pmatrix}, \quad
             p^3=\begin{pmatrix}
                -(1-\frac{D}{2}) \\ 1-\frac{D}{2} \\ -(D-1)
            \end{pmatrix}, \quad
             p^4=\begin{pmatrix}
                1-\frac{D}{2} \\ -(1-\frac{D}{2}) \\ -(D-1)
            \end{pmatrix}.
        \end{equation*}
         Inradius: $r=1-\frac{D}{2}$.\\
        Incenter: $c=(0,0,0)^{\top}$.\\
        Containment: 
        \begin{align*}
            c+ru^1&=\frac{1}{2}(p^1+p^4),\\
            c-ru^1&=\frac{1}{2}(p^2+p^3),\\
            c+ru^2&=\frac{1}{2}(p^1+p^3),\\
            c-ru^2&=\frac{1}{2}(p^2+p^4),\\
            c+ru^3&=\frac{D}{8(D-1)}(p^1+p^2)+\frac{3D-4}{8(D-1)}(p^3+p^4),\\
            c-ru^3&=\frac{3D-4}{8(D-1)}(p^1+p^2)+\frac{D}{8(D-1)}(p^3+p^4).
        \end{align*}
        Touching points: $ (1-\frac{D}{2})u^1$, $-(1-\frac{D}{2})u^1$.\\
        Outer normals: $u^1$, $-u^1$.\\

 \textit{Configuration 2:}\\
    Vertices: 
        \begin{equation*}
             p^1=\begin{pmatrix}
                1-\frac{D}{2} \\ 1-\frac{D}{2} \\ D-1
            \end{pmatrix}, \quad
             p^2=\begin{pmatrix}
               -( 1-\frac{D}{2}) \\ -(1-\frac{D}{2}) \\ D-1
            \end{pmatrix} \quad
             p^3=\begin{pmatrix}
                -(D-1) \\ 1-\frac{D}{2} \\ -(1-\frac{D}{2})
            \end{pmatrix}, \quad
             p^4=\begin{pmatrix}
               1-\frac{D}{2} \\ -(1-\frac{D}{2}) \\ -(D-1)
            \end{pmatrix}.
        \end{equation*}
          Inradius: $r=1-\frac{D}{2}$.\\
        Incenter: $c=(  \frac{(3D-4)(D-2)}{2D}, 0,\frac{(3D-4)(D-1)}{D} )^{\top}$.\\
        Containment: 
        \begin{align*}
            c+ ru^1&=\frac{1}{2}p^1+\frac{3D-4}{2D}p^2+ \frac{2-D}{D}p^4,\\
            c- ru^1&=\frac{(3D-4)^2}{2D^2}p^1+\frac{1}{2}p^2+ \frac{4(2-D)(D-1)}{D^2} p^3,\\
             c+ru^2&=\frac{5D^2-12D+8}{D^2}p^1+\frac{4(D-1)(2-D)}{D^2}p^3,\\
          c-ru^2&=\frac{2(D-1)}{D}p^2+\frac{2-D}{D}p^4,\\
           c+ru^3&=\frac{7D^2-20D+16}{2D^2}p^1+\frac12 p^2+\frac{(2-D)(3D-4)}{D^2}p^3,\\
            c-ru^3&=\frac{3D-4}{2D}p^1+\frac{3D-4}{2D}p^2+ \frac{2-D}{D}p^3+\frac{2-D}{D}p^4.
        \end{align*}
        Touching points: $c+ (1-\frac{D}{2}) u^2$, $c- (1-\frac{D}{2})u^2$. \\
        Outer normals: $u^2$, $-u^2$.\\

 \textit{Configuration 3:}\\
    Vertices:
        \begin{equation*}
             p^1=\begin{pmatrix}
                1-\frac{D}{2} \\ 1-\frac{D}{2} \\ D-1
            \end{pmatrix}, \quad
             p^2=\begin{pmatrix}
               -( 1-\frac{D}{2}) \\ -(1-\frac{D}{2}) \\ D-1
            \end{pmatrix}, \quad
             p^3=\begin{pmatrix}
                -(1-\frac{D}{2}) \\ D-1 \\ -(1-\frac{D}{2})
            \end{pmatrix}, \quad
             p^4=\begin{pmatrix}
                D-1 \\ -(1-\frac{D}{2}) \\ -(1-\frac{D}{2})
            \end{pmatrix}.
        \end{equation*}
         Inradius: $r=\frac{D(2-D)}{4-D}$.\\
        Incenter: $c=(-\frac{(3D-4)(D-2)}{8-2D}, -\frac{(3D-4)(D-2)}{8-2D}, \frac{3D-4}{4-D} )^{\top}$.\\
        Containment: 
        \begin{align*}
            c+ ru^1&=\frac{D}{4-D}p^1+\frac{4-2D}{4-D}p^4,\\
            c- ru^1&=\frac{D}{4-D}p^2+\frac{4-2D}{4-D}p^3,\\
             c+ ru^2&=\frac{D}{4-D}p^1+\frac{4-2D}{4-D}p^3,\\
            c-ru^2&=\frac{D}{4-D}p^2+\frac{4-2D}{4-D}p^4,\\
             c+ ru^3&=\frac{D}{4-D}p^1+\frac{4-2D}{4-D}p^2,\\
             c- ru^3&=\frac{3D-4}{4-D}p^2 +\frac{2(2-D)}{4-D}p^3 +\frac{2(2-D)}{4-D}p^4.
        \end{align*}
        Touching points: $c+ \frac{D(2-D)}{4-D}u^1$,  $c- \frac{D(2-D)}{4-D}u^1$, $c+ \frac{D(2-D)}{4-D}u^3$. \\ 
        Outer normals:$(D,D,5D-8)^{\top}$, $(-1,-1,-1)^{\top}$, $u^3$.\\

 \textit{Configuration 4:}\\
    Vertices:
        \begin{equation*}
             p^1=\begin{pmatrix}
                1-\frac{D}{2} \\ 1-\frac{D}{2} \\ D-1
            \end{pmatrix}, \quad
             p^2=\begin{pmatrix}
               -( 1-\frac{D}{2}) \\ -(1-\frac{D}{2}) \\ D-1
            \end{pmatrix}, \quad
             p^3=\begin{pmatrix}
                -(D-1) \\ 1-\frac{D}{2} \\ -(1-\frac{D}{2})
            \end{pmatrix}, \quad
             p^4=\begin{pmatrix}
                D-1 \\ -(1-\frac{D}{2}) \\ -(1-\frac{D}{2})
            \end{pmatrix}.
        \end{equation*}
         Inradius: $r=1-\frac{D}{2}$.\\
        Incenter: $c=( 0 , 0, \frac{3D-4}{2} )^{\top}$.\\
        Containment: 
        \begin{align*}
            c+ ru^1&=\frac{1}{2}p^1+\frac{3D-4}{2D}p^2+ \frac{2-D}{D}p^4,\\
            c- ru^1&=\frac{3D-4}{2D}p^1+\frac{1}{2}p^2+ \frac{2-D}{D} p^3,\\
             c+ ru^2&=\frac{2(D-1)}{D}p^1+ \frac{2-D}{D}p^3,\\
            c-ru^2&=\frac{2(D-1)}{D}p^2+\frac{2-D}{D}p^4,\\
             c+ ru^3&=\frac{1}{2}(p^1+p^2),\\
             c-ru^3&=\frac{3D-4}{2D}p^1+\frac{3D-4}{2D}p^2+\frac{2-D}{D}p^3+\frac{2-D}{D}p^4.
        \end{align*}
        Touching points: $c+ (1-\frac{D}{2})u^2$, $c- (1-\frac{D}{2})u^2$. \\ 
        Outer normals: $u^2$, $-u^2$.\\

 \textit{Configuration 5:}\\
    Vertices:  
        \begin{equation*}
             p^1=\begin{pmatrix}
                1-\frac{D}{2} \\ 1-\frac{D}{2} \\ D-1
            \end{pmatrix}, \quad
             p^2=\begin{pmatrix}
               -( 1-\frac{D}{2}) \\ -(D-1) \\  1-\frac{D}{2}
            \end{pmatrix}, \quad
             p^3=\begin{pmatrix}
                -( 1-\frac{D}{2}) \\ D-1 \\ -(1-\frac{D}{2})
            \end{pmatrix}, \quad
             p^4=\begin{pmatrix}
               D-1 \\ -(1-\frac{D}{2}) \\ -( 1-\frac{D}{2})
            \end{pmatrix}.
        \end{equation*}
           Inradius: $r=\frac{D(7D^2-20D+16)}{2(11D^2-24D+16)}$.\\
        Incenter: $c=\frac{3D-4}{11D^2-24D+16}( 2D^2-5D+4,\frac{-(3D-4)(2-D)}{2},  2D^2-5D+4)^{\top}$.\\

        Containment: 
        \begin{align*}
           c+ ru^1 &=\frac{D^2}{11D^2-24D+16}p^1+\frac{10D^2-24D+16}{11D^2-24D+16}p^4,\\
           c- ru^1 &=\frac{3D^2-4D}{11D^2-24D+16}p^1+\frac{D^2}{11D^2-24D+16}p^2+\frac{7D^2-20D+16}{11D^2-24D+16}p^3,\\
        c+ ru^2 &=\frac{D^2}{11D^2-24D+16}p^1+\frac{7D^2-20D+16}{11D^2-24D+16}p^3+\frac{3D^2-4D}{11D^2-24D+16}p^4,\\
           c-ru^2 &=\frac{3D^2-4D}{11D^2-24D+16}p^1+\frac{D^2}{11D^2-24D+16}p^2+\frac{7D^2-20D+16}{11D^2-24D+16}p^4,\\
            c+ ru^3 & =\frac{4D^2-4D}{11D^2-24D+16}p^1+\frac{4D-2D^2}{11D^2-24D+16}p^2+\frac{9D^2-24D+16}{11D^2-24D+16}p^4,\\
            c-ru^3 &=\frac{3D^2-4D}{11D^2-24D+16}p^2+\frac{7D^2-20D+16}{11D^2-24D+16}p^3+\frac{D^2}{11D^2-24D+16}p^4.
        \end{align*}
     
        Touching points: $c-ru^1$,  $c+ru^2$,  $c-ru^2$,  $c-ru^3$. \\
        Outer normals: $(\frac{1}{D-2},\frac{4-2D}{5D^2-12D+8},\frac{4D-4}{5D^2-12D+8})^{\top}$, $(D,D,5D-8)^{\top}$, $(1,-1,1)^{\top}$, $(1,1,\frac{2D-2}{2-D})^{\top}$. \\
     
         \textit{Configuration 6:}\\
 Vertices:
        \begin{equation*}
            p^1=\begin{pmatrix}
                1-\frac{D}{2}  \\ 1-\frac{D}{2} \\ D-1
            \end{pmatrix}, \quad
             p^2=\begin{pmatrix}
               -(1-\frac{D}{2}) \\ -(D-1) \\ 1-\frac{D}{2}
            \end{pmatrix}, \quad
             p^3=\begin{pmatrix}
                -(1-\frac{D}{2}) \\ D-1 \\ -(1-\frac{D}{2})
            \end{pmatrix}, \quad
             p^4=\begin{pmatrix}
                1-\frac{D}{2} \\ -(1-\frac{D}{2}) \\ -(D-1)
            \end{pmatrix}.
        \end{equation*}
          Inradius: $r=1-\frac{D}{2}$.\\
        Incenter: $c=(0,0,0)^{\top}$.\\
  Containment: 
  \begin{align*}
      c+ru^1&=\frac{1}{2}(p^1+p^4),\\
       c-ru^1&=\frac{1}{2}(p^2+p^3),\\
        c+ru^2&=\frac{7D^2-20D+16}{4(5D^2-12D+8)} p^1 + \frac{(4-3D)^2}{4(5D^2-12D+8)}p^2\\&+\frac{D^2}{4(5D^2-12D+8)}p^3+\frac{D(3D-4)}{4(5D^2-12D+8)}p^4,\\
         c-ru^2&=\frac{D(3D-4)}{4(5D^2-12D+8)} p^1 + \frac{D^2}{4(5D^2-12D+8)}p^2\\&+\frac{(4-3D)^2}{4(5D^2-12D+8)}p^3+\frac{7D^2-20D+16}{4(5D^2-12D+8)}p^4,\\
  \end{align*}
  \begin{align*}
            c+ru^3&=\frac{D^2}{4(5D^2-12D+8)} p^1 + \frac{7D^2-20D+16}{4(5D^2-12D+8)}p^2\\&+\frac{D(3D-4)}{4(5D^2-12D+8)}p^3+\frac{(4-3D)^2}{4(5D^2-12D+8)}p^4,\\
            c-ru^3&=\frac{(4-3D)^2}{4(5D^2-12D+8)}p^1+\frac{D(3D-4)}{4(5D^2-12D+8)}p^2\\ &+\frac{7D^2-20D+16}{4(5D^2-12D+8)}p^3+\frac{D^2}{4(5D^2-12D+8)}p^4.    
  \end{align*}
        Touching points: $ (1-\frac{D}{2})u^1$, $-(1-\frac{D}{2})u^1$.\\
        Outer normals: $u^1$, $-u^1$.\\
\end{proof}

\end{document}